\documentclass{article}

\usepackage{slsec}
\usepackage{stud_log}
\usepackage{slfoot}
\usepackage{slthm}

\usepackage[english]{babel}
\usepackage[T1]{fontenc}

\usepackage{amsmath}
\usepackage{amssymb}

\usepackage{enumerate}
\usepackage{xcolor}

\def\R{\RW}

\def\N{\mathbb{N}}  
\def\oNum{n}    

\def\J{\mathcal{L}} 
\def\oSen{p}    

\def\M{\mathfrak{M}}    

\def\eVal{v}    

\def\R{\text{R}}    
\def\cRel{\mathbf{R}}   
\def\cRelsub{\mathbf{Q}}    

\def\fDem{d}    

\def\cPos{\mathsf{W}}   
\def\oPos{\mathsf{w}}  

\def\rAvail{\mathsf{Q}} 

\def\RModal{\R_{\mathsf{w}}} 
\def\cRModal{\{ \RModal \}_{\oPos \in \cPos}}   

\def\A{A} 
\def\B{B} 
\def\C{C} 

\def\Var{\mathsf{Var}}  
\def\For{\mathsf{For}}  

\def\CPL{\textbf{CPL}}  
\def\BCL{\textbf{BCL}}  
\def\BCLCUN{\textbf{BCL}^{\neg}}   
\def\MBCL{\textbf{MBCL}}    
\def\MBCLCUN{\textbf{MBCL}^{\neg}} 

\def\lGen{\J_{\textsf{Gen}}}    
\def\fGen{\For_{\textsf{Gen}}}  
\def\acGen{Ax_{\textsf{Gen}}}   
\def\cGenMax{\mathbf{Max}_{\textsf{Gen}}}   
\def\fCPL{\For_{\textsf{CPL}}}

\def\lBCL{\J_{\textsf{BCL}}}    
\def\fBCL{\For_{\textsf{BCL}}}  

\def\acBCL{Ax_{\textsf{BCL}}}   
\def\acBCLCUN{Ax_{\textsf{BCL}}^{c}}    
\def\acBCLGCUN{Ax_{\textsf{BCL}}^{\textsf{gc}}} 

\def\cBCLGCUNMax{\mathbf{Max}_{\acBCLGCUN}} 

\def\cModBCL{\mathbf{M}}    

\def\cRelBCL{\mathbf{BC}}    
\def\cRelBCLCUN{\cRelBCL^{c}}   
\def\cRelBCLGCUN{\cRelBCL^{gc}} 

\def\lMBCL{\J_{\textsf{MBCL}}}  
\def\fMBCL{\For_{\textsf{MBCL}}}    

\def\acMBCL{Ax_{\textsf{MBCL}}} 
\def\acMBCLGCUN{Ax_{\textsf{MBCL}}^{\textsf{gc}}} 
\def\acMBCLCUNDR{Ax_{\textsf{MBCL}}^{\textsf{DemR}}}    
\def\acMBCLCuNDL{Ax_{\textsf{MBCL}}^{\textsf{DemL}}}    
\def\acMBCLCUNDE{Ax_{\textsf{MBCL}}^{\textsf{DemE}}} 

\def\cMBCLGCUNMax{\mathbf{Max}_{\acMBCLGCUN}}   
\def\cMBCLCUNDLMax{\mathbf{Max}_{\acMBCLCuNDL}} 

\def\cModMBCL{\mathbf{M}} 
\def\cRelMBCL{\mathbf{F}_{B}} 
\def\cRelMBCLCUN{\mathbf{F}_{B^{c}}} 
\def\cRelMBCLGCUN{\mathbf{F}_{B^{gc}}}   
\def\cRelMBCLDEML{\mathbf{F}_{B}^{\textsf{DemL}}}   

\def\aAi{(\mathsf{A1})}                                 
\def\aAii{(\mathsf{A2})}                                
\def\aBi{(\mathsf{B1})}                                 
\def\aBii{(\mathsf{B2})}                                
\def\aImp{(\mathsf{Imp})}                               
\def\aCPL{(\mathsf{CPL})}                               
\def\acDS{(\mathsf{DS}_{\supset})}                      
\def\aDS{(\mathsf{MP}_{\rightarrow})}                   
\def\anec{(\mathsf{Nec})}                               

\def\aCUNi{(\mathsf{C1})}                                 
\def\aCUNii{(\mathsf{C2})}                                

\def\aGCUN{(\mathsf{GC1})}                             
\def\aGCUNii{(\mathsf{GC2})}   

\def\aDual{(\mathsf{Dual})}
\def\aKm{(\mathsf{K}^{\supset})}
\def\aK{(\mathsf{K})}

\def\aCUNDR{(\mathsf{CUDR})}
\def\aCUNDL{(\mathsf{CUDL})}
\def\aCUNDE{(\mathsf{CUDE})}

\def\aDOne{(\mathsf{D1})}
\def\aDTwo{(\mathsf{D2})}
\def\aT{(\mathsf{T})}
\def\aB{(\mathsf{B})}
\def\aD{(\mathsf{D})}
\def\aFour{(\mathsf{4})}
\def\aFive{(\mathsf{5})}

\def\scAi{(\mathsf{a1})}                                
\def\scAii{(\mathsf{a2})}                               
\def\scBo{(\mathsf{b0})}                                
\def\scBi{(\mathsf{b1})}                                
\def\scBii{(\mathsf{b2})}                               

\def\scCUN{(\mathsf{cun})}                              

\def\scGCUN{(\mathsf{gcun})}                            

\def\scCUNk{(\mathsf{cun}^{klmn})}      

\def\scDemR{(\mathsf{DemR})}                            
\def\scDemL{(\mathsf{DemL})}                            
\def\scDemE{(\mathsf{DemE})}                             

\def\scDOne{(\mathsf{d1})}
\def\scDTwo{(\mathsf{d2})}
\def\scKOne{(\mathsf{k1})}
\def\scKTwo{(\mathsf{k2})}
\def\scT{(\mathsf{t})}
\def\scB{(\mathsf{b})}
\def\scD{(\mathsf{d})}
\def\scFour{(\mathsf{iv})}
\def\scFive{(\mathsf{v})}

\DeclareRobustCommand\proves{\mathrel{|}\joinrel\mkern-.5mu\mathrel{-}}
\renewcommand{\vdash}{\proves}

\makeatletter
\newcommand{\labeltext}[2]{%
	\@bsphack
	\csname phantomsection\endcsname 
	\def\@currentlabel{#1}{\label{#2}}%
	\@esphack
}
\makeatother

\usepackage{tikz}
\usetikzlibrary{arrows} 
\usetikzlibrary{patterns}

\usepackage{float}

\renewcommand{\theequation}{\Roman{equation}}

\newtheorem{tw}{Theorem}[section]
\newtheorem{lm}[tw]{Lemma}
\newtheorem{wn}[tw]{Corollary}
\newtheorem{fa}[tw]{Fact}
\newtheorem{df}[tw]{Definition}

\theoremstyle{definition}

\def\ER{\textbf{R}}
\def\ES{\textbf{S}}

\begin{document}
	
	\setcounter{page}{1}
	
	\fourAuthorsTitle{T.\ts Jarmu\.zek}{J. \ts Malinowski}{A.\ts Parol}{ N.\ts Zamperlin}{Axiomatization of Boolean Connexive Logics with syncategorematic negation and modalities}

\noindent\textbf{Abstract}: In the article we investigate three classes of extended Boolean Connexive Logics. Two of them are extensions of Modal and non-Modal Boolean Connexive Logics with a property of closure under an arbitrary number of negations. The remaining one is an extension of Modal Boolean Connexive Logic with a property of closure under the function of demodalization. In our work we provide a formal presentation of mentioned properties and axiom schemata that allow us to incorporate them into Hilbert-style calculi. The presented axiomatic systems are provided with proofs of soundness, completeness, and decidability. The properties of closure under negation and demodalization are motivated by the syncategorematic view on the connectives of negation and modalities, which is discussed in the paper.

\vspace{0.2cm}


\section*{Introduction}
\label{sec:1}

In our paper, we present axiomatizations of three classes of Boolean Connexive Logics. Starting from non-Modal Boolean Connexive Logics  (in short: $\BCL$) and Modal Boolean Connexive Logics (in short: $\MBCL$), we introduce: (i) $\BCL$ closed under generalized negation, (ii) $\MBCL$ closed under generalized negation, and (iii) $\MBCL$ closed under the function of demodalization. These results are also a contribution to the fast growing area of relating logics and its applications. 

Boolean Connexive Logics were introduced in two main papers of Jarmu\.zek and Malinowski \cite{JarmuzekMalinowski2019}, \cite{JarmuzekMalinowski2019a}, and then accepted as one of the solutions to the problem of connexivity \cite{HWOnline}. Mentioned systems were constructed as an application  of relating semantics to the subject of connexive logics. This area of research resulted in more papers such as \cite{MalinowskiRelating}, \cite{MalinowskiBarber}, where Lewis Carroll's \textit{Barbershop Paradox} was analyzed by the means of $\BCL$, or \cite{KlonEG} where in the context of $\BCL$ connexivity and content-relationship were studied. In \cite{Klonowski2}, for some classes of Boolean connexive logics adequate axiomatizations were provided. In the current paper we will incorporate these recent results.

In the first section we will introduce the philosophical motivation behind Boolean Connexive Logics: \textit{minimal change strategy}. Then, we will provide a short description of the syntax and semantics for $\BCL$ and the axioms proposed in \cite{Klonowski2}. Analogously, we will introduce $\MBCL$.

In the second section we deal with the issue of  syncategorematic connectives, particularly negation and modalities. Based on that, we will define  the properties on being closed under multiple negations and under the function of demodalization. This will serve us to define extensions of $\BCL$ and $\MBCL$.

The third section is centered around the topic of closure under multiple negations. The property of being closed under negation that appeared already in the previous works, could be generalized to any number of negations. Hence we will provide a formal description of the mentioned property and introduce axiom schemata capable of expressing that property. Using these axioms we will construct the Hilbert-style systems of $\BCL$ and $\MBCL$ closed under particular instances of multiple negations. Finally, soundness and completeness will be proved.  

The fourth section, devoted to the studies of closure under function of demodalization, will have the corresponding structure. With the established  idea behind the function of demodalization and formal definition of the function, we will construct the axiom system and prove its soundness and completeness in the same way as it was done in the previous section. At last, decidability for all the introduced calculi will be proved.

\section{An introduction to Boolean Connexive Logics}
\label{sec:2}

$\BCL$ appeared for the first time in the work of Malinowski and Jarmu\.zek (see \cite{JarmuzekMalinowski2019}) as an example of application of relating semantics to connexive logics. Philosophically, its creation was laid upon the idea of \textit{minimal change strategy}, which is a similar idea to Occam's Razor. In short, the application of this idea is that instead of  getting rid of some classical laws, or changing the interpretation of more than one logical connective, it is better to change only the meaning of non-Boolean connectives. According to this approach, a common semantic fundamentals for both $\CPL$ (Classical Propositional Logic) and connexive logic is provided: relating semantics. This strategy allows us to preserve all $\CPL$ tautologies and changes the interpretation of only one logical connective, namely implication. While the rest of the symbols preserve their Boolean interpretation, the implication has a relating, non-extensional meaning, therefore this operation is no longer material implication. Minimal change strategy in application to $\BCL$ could be summed up by the following aphorism: 

\begin{quote}
	``What is Boolean remains Boolean ($\neg, \wedge, \vee$), what is not Boolean ($\rightarrow$) becomes connexive.''
\end{quote}

Due to the fact that the rest of the connectives are interpreted in the standard way, material implication could be incorporated as a derived symbol, defined by means of other Boolean connectives. Connexive implication in $\BCL$ is interpreted using  extensional truth values with the addition of an intensional factor - the relating relation. This relation is intended to express the connection between two expressions. So, instead of saying that two sentences are related, we can specify the way in which they are related, and that means they are \textit{connected} (in some sense of connexivity). This way in relating semantics the very high level of connexivity is preseved. This was the core idea behind the creation of Boolean Connexive Logic. 

The language of $\BCL$, $\lBCL$, consists of the same symbols as the language of $\CPL$: variables: $\Var = \{ \oSen_{\oNum} \colon \oNum \in \N \}$,  one unary connective $\neg$, three binary connectives $\wedge, \vee, \rightarrow$ and brackets  ), (. On the other hand, the language of $\MBCL$,   $\lMBCL$,  is just an extension of $\lBCL$ with unary modal operators $\Box, \Diamond$. The set of formulas of $\lBCL$ is defined in the standard way and is denoted as $\fBCL$. Analogously, $\fMBCL$ denotes formulas of $\lMBCL$. Obviously, $\fCPL = \fBCL \subset \fMBCL$. Let $\A, \B \in \fBCL$ (for $\BCL$) or $\A, \B \in \fMBCL$ (for $\MBCL$). In  both languages we will use the following abbreviations for, respectively, material implication and material equivalence:
\begin{description}
	\item ($\supset$)  $\A \supset \B := \neg \A \vee \B$
	\smallskip
	\item ($\equiv$)   $\A \equiv \B := (\neg \A \vee \B) \wedge (\neg \B \vee \A)$.
\end{description}

The main feature of any connexive logic is the validity of Aristotle's and Boethius' theses. The following four expressions are schemata of those theses in the language of Boolean Connexive Logics. 
\medskip
\begin{description}
	\item[$\aAi$] $\neg (\neg \A \rightarrow \A)$
	\smallskip
	\item[$\aAii$] $\neg (\A \rightarrow \neg \A)$
	\smallskip
	\item[$\aBi$] $(\A \rightarrow \B) \rightarrow \neg (\A \rightarrow \neg \B)$ 
	\smallskip
	\item[$\aBii$] $(\A \rightarrow \neg  \B) \rightarrow \neg (\A \rightarrow \B)$
\end{description}

\noindent Moreover the arrow $\to$ is required not to be symmetric, that is $(A \to B) \supset (B \to A)$ should not be a theorem of any connexive logic\footnote{Check \cite{HWOnline} for a definition of the properties that are widely accepted as characterizing connexivity and an extensive introduction to the topic.}.

Semantics for $\BCL$ is pretty simple, due to the lack of interpretation of modal symbols, thus it will be described firstly. A model for $\lBCL$ is defined as an ordered pair of a valuation function and a binary relation on formulas with the following truth-conditions. 

\begin{df}[Model for $\lBCL$]\label{ModelBCL}
	A model for $\lBCL$ (\textit{based on relation} $\R$) is an ordered pair $\langle v,\R \rangle$, where:
	\begin{description}
		\item $\bullet$  $\eVal \colon \Var \longrightarrow \{1,0\}$ is a valuation of variables
		\smallskip
		\item $\bullet$ $\R \subseteq \fBCL \times \fBCL$ is a binary relation between formulas.
	\end{description}
\end{df}

We assume the following truth-conditions:

\begin{df}[Truth-conditions in model for $\lBCL$]\label{InterpretationBCL}
	Let $\A \in \fBCL$. $A$ is \textit{true} in $\M = \langle \eVal, \R \rangle$ (in short: $\M \vDash \A$; $\M \not \vDash \A$, if it is not the case) iff for all $\B, \C \in \fBCL$:
	\begin{description}
		
		\item $(\Var)$ $\eVal(\A) = 1$, if $\A \in \Var$
		\smallskip
		\item $(\neg)$ $\M \nvDash \B$, if $\A = \neg \B$
		\smallskip
		\item $(\wedge)$ $\M \vDash \B$ and $\M \vDash \C$, if $\A = \B \wedge \C$
		\smallskip
		\item $(\vee)$ $\M \vDash \B$ or $\M \vDash \C$, if $\A = \B \vee \C$
		\smallskip
		\item $(\rightarrow)$ $[\M \nvDash \B$ or $\M \vDash \C]$ and $\R(\B,\C)$, if $\A = \B \rightarrow \C$.
	\end{description}
\end{df}

\n As we can see from the previous definition, the implication ($\rightarrow$) has a non-classical interpretation in $\BCL$. One part of the truth condition is an assignment of values to its components, and the second part is a relation between those components. 

In the definition of model for $\lBCL$ the set of formulas $\fBCL$ plays an active role, since $\R$ is a relation over $\fBCL$. It could be argued that a more exhaustive presentation of a model should include the set of formulas, obtaining the tuple $\langle \fBCL, \eVal, \R \rangle$. Despite that, we will keep omitting $\fBCL$, because the language $\lBCL$ is fixed and there is no risk of confusion.

Let us denote the class of all  relations $\R \subseteq \fBCL \times \fBCL$ as $\cRel$, and to denote subsets of $\cRel$ we will use $\cRelsub, \cRelsub_{1}, \cRelsub_{2},...$ respectively. Then we will denote the class of all models for $\fBCL$ by $\cModBCL{\cRel}$, and by $\cModBCL{\cRelsub}$ the class of all models based on some subset $\cRelsub \subseteq \cRel$.  The same notation will be used for $\fMBCL$, if it is not misleading. So, when we talk about $\MBCL$, $\cRel$ will denote $\fMBCL \times \fMBCL$ etc. Now can introduce the notion of validity.

\begin{df}[Validity for $\fBCL$]\label{ValidityBCL}
	Let $\A \in \fBCL$. $\A$ is \textit{valid} in a relation $\R \in \cRel$ (in short: $\R \vDash \A$) iff for all valuations $\eVal$, $\langle \eVal, \R \rangle \vDash \A$. Similarly, $\A$ is \textit{valid} in a class of relations $\cRelsub \subseteq \cRel$ (in short: $\vDash_{\cRelsub} \A$) iff for all relations $\R \in \cRelsub$, $\R \vDash \A$.
\end{df}

Due to the fact that $\BCL$ is a connexive logic, it is crucial for Aristotle's and Boethius' laws to be satisfied by the models. Henceforth, the relating relation has to be limited by some specific conditions. We will use the abbreviation $\sim \R(\A,\B)$ when $\R(\A,\B)$ does not hold. Let $\A, \B \in \fBCL$, and $\R \in \cRel$. In \cite{JarmuzekMalinowski2019} the following properties were defined: 

\begin{description}
	\item[$\scAi$] $\sim \R( \A, \neg \A)$
	\smallskip
	\item[$\scAii$] $\sim \R(\neg \A, \A)$
	\smallskip
	\item[$\scBo$] $\R(\A,\B) \Rightarrow \, \sim \R(\A, \neg \B)$
	\smallskip
	\item[$\scBi$] $\R(\A \rightarrow \B, \neg (\A \rightarrow \neg \B))$
	\smallskip
	\item[$\scBii$] $\R(\A \rightarrow \neg \B, \neg (\A \rightarrow \B))$
	\smallskip
	\item[$\scCUN$] $\R(\A, \B) \Rightarrow \R(\neg \A, \neg \B)$.
	\smallskip
\end{description}

\n The conditions $\scAi$ and $\scAii$ correspond to Aristotle's theses, whereas the combination of conditions $\scBo$, $\scBi$ and $\scBii$ correspond to Boethius' theses. Let us denote the class of all relations that satisfy those conditions as $\cRelBCL$ and the class of all relations satisfying stated conditions with the addition of $\scCUN$ as $\cRelBCLCUN$.

The mentioned correspondence between theorems $\aAi$, $\aAii$, $\aBi$, $\aBii$ and $\scAi$, $\scAii$, $\scBo$, $\scBi$, $\scBii$ was proved in \cite{JarmuzekMalinowski2019}. Two theorems state that if the class of relations satisfy conditions $\scAi$, $\scAii$, $\scBo$, $\scBi$, $\scBii$ it is the case that Aristotle's and Boethius' theses are valid in such structure\footnote{Thanks to the transparent behaviour of the relating relation $\R$, it is easy to implement variations of connexivity. E.g., if we want a semantics which validates not the full Boethius' theses but their weaker version in the form of $(A \to B) \supset \neg(A \to \neg B)$ and $(A \to \neg B) \supset \neg(A \to B)$, it is enough to retain $\scBo$ while rejecting $\scBi$ and $\scBii$, respectively. Similarly, the so-called secondary Boethius' theses $(A \to B) \supset \neg(\neg A \to B)$ and $(\neg A \to B) \supset \neg(A \to B)$ can be validated changing $\scBo$ for ($\mathsf{b0'}$) $\R(A,B) \Rightarrow \sim\R(\neg A,B)$, and rewriting $\scBi,\scBii$ as ($\mathsf{b1'}$) $\R(A \to B,\neg(\neg A \to B))$ and ($\mathsf{b2'}$) $\R(\neg A \to B,\neg(A \to B))$.}. But the backward implication is not true, unless all the relations in the class satisfy condition $\scCUN$.

\begin{tw}[Theorem 5.1 from \cite{JarmuzekMalinowski2019}]\label{Cor1BCL}
	Let $\R \in \cRel$, then:
	\begin{description}
		\item[(a)] $\R$ satisfies $\scAi \Rightarrow \R \vDash \neg (\A \rightarrow \neg \A)$
		\smallskip
		\item[(b)] $\R$ satisfies $\scAii \Rightarrow \R \vDash \neg (\neg \A \rightarrow \A)$
		\smallskip
		\item[(c)] $\R$ satisfies $\scBo$ and $\scBi \Rightarrow \R \vDash (\A \rightarrow \B) \rightarrow \neg (\A \rightarrow \neg \B)$
		\smallskip
		\item[(d)] $\R$ satisfies $\scBo$ and $\scBi \Rightarrow \R \vDash (\A \rightarrow \neg \B) \rightarrow \neg(\A \rightarrow \B)$.
	\end{description}
\end{tw}

Theorem \ref{Cor1BCL} can be summed up by the following fact. The proof is obvious due to the definition of the class $\cRelBCL$.

\begin{fa}\label{FactCorr1BCL}
	If  $\R \in \cRelBCL$ then $\aAi, \aAii, \aBi$, $\aBii$ are valid in $\R$.
\end{fa}




\begin{tw}[Theorem 6.1 from \cite{JarmuzekMalinowski2019}]\label{CorrespondenceTwoBCL}
	Let $\R \in \cRel$ and $\R$ satisfies $\scCUN$ then: 
	\begin{description}
		\item[(a)] $\R$ satisfies $\scAi \Leftrightarrow \R \vDash \neg (\A \rightarrow \neg \A)$
		\smallskip
		\item[(b)] $\R$ satisfies $\scAii \Leftrightarrow \R \vDash \neg (\neg \A \rightarrow \A)$
		\smallskip
		\item[(c)] $\R$ satisfies $\scBo$ and $\scBi \Leftrightarrow \R \vDash (\A \rightarrow \B) \rightarrow \neg (\A \rightarrow \neg \B)$
		\smallskip
		\item[(d)] $\R$ satisfies $\scBo$ and $\scBi \Leftrightarrow \R \vDash (\A \rightarrow \neg \B) \rightarrow \neg(\A \rightarrow \B)$.
	\end{description}
\end{tw}

Theorem \ref{CorrespondenceTwoBCL} can be summed up by the following fact. The proof is obvious due to the definition of the class $\cRelBCLCUN$.

\begin{fa}\label{CorrespondenceFactTwoBCL} $\R \in \cRelBCLCUN$ iff $\aAi$, $\aAii$, $\aBi$, $\aBii$ are valid in $\R$.
\end{fa}

The classes of models $\cModBCL{\cRelBCL}$ and $\cModBCL{\cRelBCLCUN}$ define the least $\BCL$ and the least $\BCL$ closed under negation respectively (in short: $\BCLCUN$). This became clear when axiom systems were introduced for these logics. The systems were published in \cite{Klonowski2}. For the logic characterized by $\cModBCL{\cRelBCL}$ it was assumed: 
\begin{description}
	\item[$\aCPL$] $\A$, where $\A$ is a substituition instance of a classical propositional tautology
	\smallskip
	\item[$\aAi$] $\neg (\A \rightarrow \neg \A)$
	\smallskip
	\item[$\aAii$] $\neg (\neg \A \rightarrow \A)$
	\smallskip
	\item[$\aBi$] $(\A \rightarrow \B) \rightarrow \neg (\A \rightarrow \neg \B)$
	\smallskip
	\item[$\aBii$] $(\A \rightarrow \neg \B) \rightarrow \neg (\A \rightarrow \B)$
	\smallskip
	\item[$\aImp$] $(\A \rightarrow \B) \supset (\A \supset \B)$
	\smallskip
\end{description}

\n And the disjunctive syllogism:

\begin{align*}
	\acDS \quad\quad \begin{array}{l}
		A  \\
		A \supset B \\ \hline
		B
	\end{array}
\end{align*}

\n The described axiom system will be denoted by $\acBCL$.  Whereas the system with addition of below axiom schemata for logic defined by $\cModBCL{\cRelBCLCUN}$ will be denoted as $\acBCLCUN$:

\begin{description}
	\item[$\aCUNi$] $(\A \rightarrow \B) \supset ((\neg \A \rightarrow \neg \B) \vee (\neg \A \wedge \B))$
	\smallskip
	\item[$\aCUNii$] $(\A \rightarrow \B) \supset (\neg \neg \A \rightarrow \neg \neg \B)$
\end{description}

\n Since $\acBCL$ contains additionally all classical tautologies, both systems are  closed under the Modus Ponens  for relating implication (by $\aImp$ and $\acDS$): 

\begin{align*}
	\aDS \quad\quad \begin{array}{l}
		A  \\
		A \rightarrow B \\ \hline
		B
	\end{array}
\end{align*}


It becomes apparent from the conditions imposed over the relating relation $\R$ that it is a minimal relation, in the sense that it is required to satisfy the properties that guarantee connexivity and nothing else. Some very weak properties do not hold in general for $\R$, like reflexivity and transitivity, which, anticipating the completeness results of theorem \ref{completeness_basic_bcl}, can be easily seen to fail for the the logics $\acBCL$ and $\acBCLCUN$ in the form of the lack of the theorems $A \to A$ and $((A \to B) \land (B\to C)) \supset (A \to C)$ in the calculi. There is a precise choice behind this unusual behavior of the connexive arrow. The enterprise of Boolean Connexive Logic has so far focused on finding minimal conditions capable of capturing connexivity within the framework of relating semantics. This is achieved with $\acBCL$, the minimal \BCL, which contains the four connexive laws and, as completeness makes apparent, invalidates symmetry for the connexive arrow, that is $(A \to B) \supset (B \to A)$ is not a theorem of the calculus. Once that foundational point has been consolidated, the focus shifts to the possibility of building upwards in the direction of stronger systems. Adding properties of $\R$ which are easy to express syntactically, like the two mentioned above, amounts to provide axiomatic extensions of said minimal systems which were the focus of the preliminary study. In a similar fashion, the systems that we will introduce in the current work are meant to capture the minimal changes that are required from the starting point of $\BCL$ (and in the following \MBCL) in order to achieve a generalized closure under negation (and demodalization for \MBCL). In this sense the focus here is more the collocation of our systems within the field of relating semantics and how they may shed a light on the interplay between conditions over the relating relation and the associated logical systems, instead of trying to provide a philosophical discussion over the nature of connexivity. For this reason we prefer to stick to the weakest systems that allow to reach our goals.

Semantics for $\MBCL$ is more complicated, due to the fact that the relating relation has to be associated to each possible world. Hence, relational structure has to be combined with a modal structure in the definition of a model for $\lMBCL$. Such structures are of the form $\langle \cPos, \rAvail, \cRModal \rangle$, which will be called \textit{combined frames}, or just \textit{frames} for simplicity. The combined frame is composed of a modal frame $\langle \cPos, \rAvail \rangle$, and of a relating frame $\cRModal \subseteq \{ \R_{\oPos} \colon \R \in \cRel = \fMBCL \times \fMBCL, \oPos \in \cPos \}$.

\begin{df} [Model for $\lMBCL$]\label{ModelMBCL}
	A model for $\lMBCL$ is an ordered quadruple $\langle \cPos, \rAvail, \cRModal, \eVal \rangle$ such that:
	\begin{description}
		\item $\bullet$ $\cPos \neq \emptyset$
		\smallskip
		\item $\bullet$ $\rAvail \subseteq \cPos \times \cPos$
		\smallskip
		\item $\bullet$ $\cRModal$ is a family of relations, $\RModal \subseteq \fMBCL \times \fMBCL$ for any $\oPos \in \cPos$
		
		\item $\bullet$ $\eVal \colon \cPos \times \Var \longrightarrow  \{1,0\}$ is a valuation of variables.
	\end{description}
\end{df}

\begin{df}[Truth-conditions in model for $\lMBCL$]\label{InterpretationMBCL}
	Let $\A\in\fMBCL$. Let $\M = \langle \cPos, \rAvail, \cRModal, \eVal \rangle$ and $\oPos \in \cPos$. $A$ is \textit{true} in $\M, \oPos$ (in short: $\M, \oPos \vDash \A$; $\M, \oPos \not \vDash \A$, if it is not the case) iff for all $\B ,\C \in \fMBCL$:
	\begin{description}
		\item $(\Var)$ $\eVal (\oPos, \A) = 1$, if $\A \in \Var$
		\smallskip
		\item $(\neg)$   $\M, \oPos \nvDash \B$, if $\A = \neg \B$
		\smallskip
		\item  $(\wedge)$   $\M, \oPos \vDash \B$ and $\M, \oPos \vDash \C$, if $\A = \B \wedge \C$
		\smallskip
		\item $(\vee)$  $\M, \oPos \vDash \B$ or $\M, \oPos \vDash \C$, if $\A = \B \vee \C$
		\smallskip
		\item $(\Box)$   $\forall_{\mathsf{u} \in \cPos} (\rAvail(\oPos, \mathsf{u} ) \Rightarrow \M, \mathsf{u} \vDash \B)$, if $\A = \Box \B$
		\smallskip
		\item$ (\Diamond)$  $\exists_{\mathsf{u} \in W} (\rAvail(\oPos, \mathsf{u})$ and $\M, \mathsf{u} \vDash \B)$, if $\A = \Diamond \B$
		\smallskip
		\item $(\rightarrow)$  $(\M, \oPos \nvDash \B$ or $\M, \oPos \vDash \C)$ and $\RModal(\B,\C)$, if $\A = \B \rightarrow \C$.
	\end{description}
\end{df}

\begin{df}[Validity for $\fMBCL$]\label{ValidityMBCL}
	
	Let $\A \in \fMBCL$. $\A$ is \textit{valid} in a frame $\langle \cPos, \rAvail, \cRModal \rangle$ (in short: $\langle \cPos, \rAvail, \cRModal \rangle \vDash \A$) iff for all valuations $\eVal$ and all worlds $\oPos \in \cPos$, $\langle \cPos, \rAvail, \cRModal, \eVal \rangle,  \oPos \vDash \A$. Similarly, $\A$ is \textit{valid} in a class of combined frames \textbf{F} (in short: $\vDash_{\textbf{F}} \A$), iff for every frame $\langle \cPos, \rAvail, \cRModal \rangle \in \textbf{F}$, $\langle \cPos, \rAvail, \cRModal \rangle \vDash A$.
\end{df}

In analogy to $\BCL$, in $\MBCL$ we can provide the list of conditions that have to be met by  relating frames for Aristotle's and Boethius' laws to be valid. These conditions are the already mentioned $\scAi$, $\scAii$, $\scBo$, $\scBi$, $\scBii$, noticing that now the relation symbol $\R$ is indexed by some world $\oPos$, so $\R_{\oPos}$\footnote{Strictly speaking, all these conditions are formulated in $\fBCL$ and they should be translated into the richer language $\fMBCL$. Since this abuse of notation is harmless, we will keep using the same names for the new conditions in the modal setting.}.

One might ask whether the interpretation of $\rightarrow$ in relating semantics, especially in the context of the modal operators $\Box$ and $\Diamond$, is not "syntax in disguise", because some of the modal principles require that there is a relation between the corresponding formulas in the worlds. E.g., the validity of $\Box A \rightarrow \Box \Box A$ will require (a) that $\R_\oPos(\Box A, \Box \Box A)$ and (b) that the modal frame is transitive (see~p.~12). Although this objection can be answered in several ways, we will consider two responses here.

First, using models containing a relation defined on pairs of formulas is an accepted methodology for relating logic. Various types of such models (also with possible worlds) are presented, for example, in \cite{TJ2020}, examples of use and related problems in \cite{TJiFP2021}, while the history of such approaches is described in \cite{klonowski2021b}. So it is a matter of an aesthetic approach rather than a technical one.

Secondly, in relating semantics the basic idea behind the interpretation of a relating connective is that the logical value of a given complex proposition, with the relating connective as the main connective, is the result of two things:  (i) the logical values of the main components of this complex proposition, supplemented with (ii) a valuation of the relation between these components.  The latter element is a formal representation of an intensional relation that emerges from the connection of several simpler propositions into one more complex proposition. So it is worth noting that the relation $\R_{\oPos}$ could be replaced by a valuation function $v_{\oPos}^{\R} \colon \fMBCL \times \fMBCL \longrightarrow \{0, 1 \}$ which, for example,  satisfies the condition $v_ {\oPos}^{\R}(\Box A, \Box \Box A)= 1$ or other conditions. When we later introduce the function of demodalization 
$\fDem$ (pp. 24--25), that condition (and other conditions for modal principles)  by $\scDemL$ slims down to reflexivity $\R_\oPos(A, A)$ or in the functional notation to $v_ {\oPos}^{\R}(A, A)= 1$ (where in $A$ non-modal connectives occur). In the light of relating semantics this assumption just states that any formula is connected to itself. Since $v_ {\oPos}^{\R}$ can  be considered as  a function that valuates the relationships between formulas, this should no longer raise the suspicion that we are dealing with syntax, but with a second valuation in the model. A discussion on the basics of relating semantics, including the problem  of valuations instead of relations is also available in the article \cite{TJ2020}.

Let us denote the class of combined structures that satisfy conditions $\scAi$, $\scAii$, $\scBo$, $\scBi$, $\scBii$ as $\cRelMBCL$ and the class of combined structures satisfying the previous conditions with the addition of the modalized $\scCUN$ as $\cRelMBCLCUN$. Let moreover $\cModBCL{\cRelMBCL}$ be the class of all models built over frames from $\cRelMBCL$, and similarly for $\cModBCL{\cRelMBCLCUN}$. The theorems regarding the relations between $\cRelMBCL$, $\cRelMBCLCUN$ and the laws of connexive logic are analogous to the theorems for $\BCL$ (see \cite{JarmuzekMalinowski2019a}).

\begin{tw}[Theorem 3.1 from \cite{JarmuzekMalinowski2019a}]\label{CorrespondenceOneMBCL}
	Let $\cRModal$ be a relating frame, $\M = \langle \cPos, \rAvail, \cRModal, \eVal \rangle$ and $\oPos \in \cPos$ then:
	\begin{description}
		\item[(a)] $\RModal$ satisfies $\scAi \Rightarrow \M, \oPos \vDash \neg (\A \rightarrow \neg \A)$
		\smallskip
		\item[(b)] $\RModal$ satisfies $\scAii \Rightarrow \M, \oPos \vDash \neg (\neg \A \rightarrow \A)$
		\smallskip
		\item[(c)] $\RModal$ satisfies $\scBo$ and $\scBi \Rightarrow \M, \oPos \vDash (\A \rightarrow \B) \rightarrow \neg (\A \rightarrow \neg \B)$
		\smallskip
		\item[(d)] $\RModal$ satisfies $\scBo$ and $\scBi \Rightarrow \M, \oPos \vDash (\A \rightarrow \neg \B) \rightarrow \neg(\A \rightarrow \B)$.
	\end{description}
\end{tw}

Theorem \ref{CorrespondenceOneMBCL} can be summed up by the following fact. The proof is obvious due to the definition of the class $\cRelMBCL$.

\begin{fa}\label{FactCorrespondenceOneMBCL}
	If $\langle \cPos, \rAvail, \cRModal \rangle \in \cRelMBCL$ then  $\aAi, \aAii, \aBi, \aBii$ are valid in $\langle \cPos, \rAvail, \cRModal \rangle$.
\end{fa}

\begin{tw}[Theorem 4.1 from \cite{JarmuzekMalinowski2019a}]\label{CorrespondenceTwoMBCL}
	
	Let $\M = \langle \cPos, \rAvail, \cRModal, \eVal \rangle$ and $\oPos \in \cPos$. Let $\R_{\oPos}$ satisfy $\scCUN$. Then:
	\begin{description}
		\item[(a)] $\R_{\oPos}$ satisfies $\scAi \Leftrightarrow \M, \oPos \vDash \neg (\A \rightarrow \neg \A)$
		\smallskip
		\item[(b)] $\R_{\oPos}$ satisfies $\scAii \Leftrightarrow  \M, \oPos \vDash \neg (\neg \A \rightarrow \A)$
		\smallskip
		\item[(c)] $\R_{\oPos}$ satisfies $\scBo$ and $\scBi \Leftrightarrow  \M, \oPos \vDash (\A \rightarrow \B) \rightarrow \neg (\A \rightarrow \neg \B)$
		\smallskip
		\item[(d)] $\R_{\oPos}$ satisfies $\scBo$ and $\scBi \Leftrightarrow  \M, \oPos \vDash (\A \rightarrow \neg \B) \rightarrow \neg(\A \rightarrow \B)$.
	\end{description}
\end{tw}

The proof of another fact is obvious due to the definition of the class $\cRelMBCLCUN$.

\begin{fa}\label{FactCorrespondenceTwoMBCL} $\langle \cPos, \rAvail, \cRModal \rangle \in \cRelMBCLCUN $ iff $ \aAi, \aAii, \aBi, \aBii$ are valid in $\langle \cPos, \rAvail,$ $\cRModal \rangle$. \end{fa}


Facts \ref{FactCorrespondenceOneMBCL} and \ref{FactCorrespondenceTwoMBCL} sum up the theorems (\ref{CorrespondenceOneMBCL}, \ref{CorrespondenceTwoMBCL}) about the classes  of $\cRelMBCL$ and $\cRelMBCLCUN$ in the same way as corresponding facts  for $\BCL$.
The classes $\cRelMBCL$ and $\cRelMBCLCUN$ characterize the least $\MBCL$ and the least $\MBCL$ closed under negation (in short: $\MBCLCUN$) respectively. The axiom systems for $\MBCL$ was provided in the work of Klonowski (see \cite{Klonowski2}). The axioms are the same listed for $\BCL$, plus:

\begin{description}
	\item[$\aDual$] $\Diamond \A \equiv \neg \Box \neg \A$
	\smallskip
	\item[$\aKm$] $\Box(\A \supset \B) \supset (\Box \A \supset \Box \A)$
	\smallskip
\end{description}
\n and the necessitation rule:
\begin{align*}
	\anec \quad\quad \begin{array}{l}
		\vdash A  \\
		\hline
		\vdash \Box \A
	\end{array}
\end{align*}

\n We denote the described axiomatic system as $\acMBCL$.

Observe that the decision to employ both the relating implication $\to$ and the box operator $\Box$ in the language $\lMBCL$ is motivated by the independence of each of the two from the other and the remaining base of connectives. Therefore $\lMBCL$ is strictly more expressive than both the non-modal language for $\BCL$ and the standard modal language without relating connectives. We show that $\to$ and $\Box$ are semantically independent.

Suppose that there exists a formula $\phi(A) \in \lMBCL$ built without the use of $\Box$ and in which the formula $A$ occurs. Let us assume that $\phi(A)$ is semantically indistinguishable from $\Box A$, that is for every model $\M$ and $\oPos\in\cPos$, $\M,\oPos\vDash \phi(A)$ iff $\M,\oPos\vDash \Box A$. Let us consider a model $\M = \langle \cPos, \rAvail, \{\R_{\oPos}\}_{\oPos \in \cPos}, \eVal \rangle$ and a point $\oPos\in\cPos$ s.t. for some $p\in\Var, \M,\oPos\vDash \phi(p)$, therefore $\M,\oPos\vDash \Box p$ as well. Let us now build the model $\M' = \langle \cPos', \rAvail', \{\R'_{\oPos}\}_{\oPos \in \cPos'}, \eVal' \rangle$, which differs from $\M$ only for the addition of a single point $\oPos'$ s.t. $\oPos\rAvail'\oPos'$ and for all $q\in\Var, \eVal'(\oPos',q)=0$, while $\R'_{\oPos'}$ is arbitrary. In particular we have $\eVal'(\oPos',p)=0$, hence $\M',\oPos'\nvDash \Box p$. A simple induction on the complexity of $\phi(p)$ proves that for a formula $B$ whose main connective is Boolean it holds that $\M,\oPos\vDash B$ iff $\M',\oPos\vDash B$. Similarly, if $\phi(p)$ is of the form $B \to C$, since $\R_{\oPos}=\R'_{\oPos}$, again $\M,\oPos\vDash B \to C$ iff $\M',\oPos\vDash B \to C$. Therefore $\M',\oPos\vDash \phi(p)$, which contradicts the assumption that $\phi(p)$ and $\Box p$ were logically equivalent.

Now assume that we can express the formula $A \to B$ in terms of a formula $\phi(A,B)$ in which $A,B$ occur and built without the use of $\to$. Suppose that $\phi(A,B)$ is logically equivalent to $A \to B$. Let us consider a model $\M = \langle \cPos, \rAvail, \{\R_{\oPos}\}_{\oPos \in \cPos}, \eVal \rangle$ and a point $\oPos\in\cPos$ s.t. for some $p,q\in\Var, \M,\oPos\vDash \phi(p,q)$, hence by hypothesis $\M,\oPos\vDash p \to q$ follows. Let us now build the model $\M' = \langle \cPos, \rAvail, \{\R'_{\oPos}\}_{\oPos \in \cPos'}, \eVal \rangle$, which differs from $\M$ only for the fact that $\R'_\oPos = \R_\oPos\backslash\lbrace\langle p,q \rangle\rbrace$. An inductive argument similar to the previous one proves that every formula not containing $\to$ is unaffected by the change between the two models, therefore $\M',\oPos\vDash \phi(p,q)$, but since $\sim\R(p,q)$ then $\M',\oPos\nvDash p \to q$. Therefore the formulas aren't logically equivalent.

While in the literature there are well-known systems in which the modal operator can be expressed using implication and vice versa (e.g. some containment logics like Parry's, with $\Box A := (A \to A) \to A$, cf. \cite{fine86}), this is not the case for $\MBCL$, because while the former systems are provided with a strict implication whose semantic value is tied to the accessibility relation, the connexive arrow $\to$ is a relating implication, that is a material - not strict - implication with an added intensional layer, whose value is determined by the relating relation $\R$, but $\rAvail$ and $\R$ have no interplay in the minimal $\MBCL$\footnote{Neither are the extensions of $\MBCL$ presented in the following sections capable of restoring an interdefinability between $\to$ and $\Box$. In the case of $\MBCL$ closed under multiple negations (section \ref{sec:4}) it is evident that, since this axiomatic extension concerns only the interaction between negation and connexive implication, there is no contribution to the possibility of interdefinability. In the case of $\MBCL$ closed under demodalization (section \ref{sec:5}), it should be noticed that the demodalizing axioms are not capable of providing the relating arrow with an intensional layer connected to the accessibility relation, that is to say $\to$ remains a content-sensitive material implication, not a strict one. The above argument for the failure of any candidate formula which is supposed to witness the interfinability between $\to$ and $\Box$ can be adapted to both families of extensions.}.

\vspace{0.2cm}

It is worth noting that the minimal modal Boolean connexive system $\acMBCL$ has none of the principles characteristic of the logic $\aK$ or stronger normal modal logics when expressed by relating implication: 

\begin{description}
	\item[$\aDOne$] $\Diamond \A \rightarrow \neg \Box \neg \A$
	\smallskip
	\item[$\aDTwo$] $\neg \Box \neg \A \rightarrow \Diamond \A$
	\smallskip
	\item[$\aK$] $\Box(\A \rightarrow \B) \rightarrow (\Box \A \rightarrow \Box \B)$
	\smallskip 
	\item[$\aT$] $\Box\A \rightarrow \A$
	\smallskip
	\item[$\aD$] $\Box\A \rightarrow \Diamond \A$
	\smallskip
	\item[$\aB$] $\A \rightarrow \Diamond \Box\A$
	\smallskip
	\item[$\aFour$] $\Box\A \rightarrow \Box \Box\A$
	\smallskip
	\item[$\aFive$] $\Diamond\A \rightarrow \Box \Diamond\A$.
\end{description}

\n Only their weaker versions,  where we replace $\rightarrow$ with $\supset$,  hold in a class of frames once the usual properties over the accessibility relation have been imposed. In \cite{JarmuzekMalinowski2019a} two ways were presented in which we can validate these modal laws in stronger versions. The first one assumes some restrictions on modal and relating frames that we present below. The second is a demodalization move that we discuss in the section \ref{sec:3} and explore extensively in the section  \ref{sec:5}. 

The listed modal laws could be true in the model, if following properties of combined frames are fulfilled. (Their designations match the designations of the corresponding modal principles.) Let $\M = \langle \cPos, \rAvail, \eVal, \cRModal \rangle$ and $\oPos \in \cPos$, then:

\begin{description}
	\item[$\scDOne$] $\RModal(\Diamond \A, \neg \Box \neg \A)$
	\smallskip
	\item[$\scDTwo$] $\RModal(\neg \Box \neg \A, \Diamond \A)$
	\smallskip
	\item[$\scKOne$] $\RModal(\Box(\A \rightarrow \B), (\Box \A \rightarrow \Box \B))$
	\smallskip
	\item[$\scKTwo$] $\forall_{\mathsf{u} \in \cPos}((\rAvail(\oPos, \mathsf{u}) \Rightarrow \R_{\mathsf{u}}(\A, \B)) \Rightarrow \RModal(\Box \A, \Box \B))$
	\smallskip
	\item[$\scT$] $\RModal(\Box \A, \A)\textnormal{ and the modal frame (i.e. $\langle \cPos, \rAvail\rangle$) is reflexive}$
	\smallskip
	\item[$\scD$] $\RModal(\Box \A, \Diamond\A)\textnormal{ and the modal frame (i.e. $\langle \cPos, \rAvail\rangle$) is serial}$
	\smallskip
	\item[$\scB$] $\RModal(\A, \Box \Diamond\A)\textnormal{ and the modal frame (i.e. $\langle \cPos, \rAvail\rangle$) is symmetrical}$
	\smallskip
	\item[$\scFour$] $\RModal(\Box \A, \Box \Box\A)\textnormal{ and the modal frame (i.e. $\langle \cPos, \rAvail\rangle$) is transitive}$
	\smallskip
	\item[$\scFive$] $\RModal(\Diamond \A, \Box \Diamond\A)\textnormal{ and the modal frame (i.e. $\langle \cPos, \rAvail\rangle$) is Euclidean}$.
\end{description}

\n In \cite{JarmuzekMalinowski2019a} it was shown that by imposing the individual properties: $\scDOne$, $\scDTwo$, $\scKOne$, $\scKTwo$, $\scT$, $\scD$, $\scB$, $\scFour$, $\scFive$  on frames from $\cRelMBCL$ or $\cRelMBCLCUN$, we get the classes of frames that validates the modal principles: $\aDOne$, $\aDTwo$, $\aK$, $\aT$, $\aD$, $\aB$, $\aFour$, $\aFive$. It should be noted that the axiom $\aK$ corresponds to two semantic conditions: $\scKOne$, $\scKTwo$. The other axioms are semantically characterized by only one condition.

Let  $X_{1},...,X_{n} \in \{\aDOne, \aDTwo, \aK, \aT, \aD, \aB, \aFour, \aFive \}$. We assume  that $\mathbf{F}_{B,X_{1},...,X_{n}}$ denotes the class of all frames from $\cRelMBCL$ that satisfy the properties corresponding to $X_{1},...,X_{n}$. Similarly, we will denote the systems of axioms reinforced by individual modal rules: $Ax_{\textsf{MBCL}, X_{1}, \dots, X_{n}}$, where $X_{1}, \dots, X_{n}$ are designation of the modal principles. In the work of \cite{Klonowski2}, the author presented the completeness theorems for these axiom systems. 

The notion of syntactic consequence for a logic \textbf{L} is constructed in the standard way and is denoted  by $\vdash_{Ax_\text{L}}$. For readability, we will omit the subscript everywhere it will be clear which logic we mean.

\begin{tw}[Theorem 5.1 from \cite{Klonowski2}]\label{completeness_basic_bcl}
	
	Let $\A \in \fBCL$. Then:
	\begin{description}
		\item[(a)] $\vdash_{\acBCL} \A$  $\Leftrightarrow$ $\vDash_{\cRelBCL} \A$
		\smallskip
		\item[(b)] $\vdash_{\acBCLCUN} \A $  $\Leftrightarrow$ $\vDash_{\cRelBCL^{c}} \A$.
	\end{description}
\end{tw}

\begin{tw}[Theorem 5.2 from \cite{Klonowski2}]\label{completeness_basic_mbcl}
	
	Let $\A \in \fMBCL$. Then:

	\[ \vdash_{Ax_{\textsf{MBCL}, X_{1}, \dots, X_{n}} } \A \Leftrightarrow  \, \, \,  \vDash_{\mathbf{F}_{B,X_{1},...,X_{n}}} \A\] 
	
	\smallskip
	\smallskip
	
	\n for all  $X_{1}, \dots, X_{n}$ $\in \{\aDOne,$$  \aDTwo,$ $ \aK, \aT, \aD, \aB, \aFour, \aFive \}$.
\end{tw}

Before moving on, using the above completeness results we discuss what forms of replacement work in  $\MBCL$. We do not consider $\BCL$, since it is a subcase of the one at hand. Let us call two formulas $A,B$ materially equivalent in a calculus $Ax_{\text{L}}$ when $\vdash_{Ax_{\text{L}}} A \equiv B$. A formula $A$ is replaceable for another formula $B$ in case, for every formula $\phi(A)$ with $A$ occurring in it, $\vdash_{Ax_{\text{L}}} \phi(A)$ implies $\vdash_{Ax_{\text{L}}} \phi(B)$. In the case of $\MBCL$, material equivalence is not a sufficient condition for replacement, in fact it fails when it comes to the connexive arrow. 

Consider $\vdash (p \to q) \equiv ((p \to q) \land (r \lor \neg r))$ and the instance of $\aBi$ $\vdash (p \to q) \to \neg (p \to \neg q)$, with $p,q,r\in\Var$. It does not follow that $\vdash ((p \to q) \land (r \lor \neg r)) \to \neg (p \to \neg q)$, in fact it is enough to build a model $\M$ s.t. for some $\oPos\in\cPos$ and we impose $\sim\R_\oPos((p \to q) \land (r \lor \neg r),\neg (p \to \neg q))$. Therefore $\M,\oPos\nvDash ((p \to q) \land (r \lor \neg r)) \to \neg (p \to \neg q)$ despite $\M,\oPos\vDash (p \to q) \to \neg (p \to \neg q)$. Instead if we switch material equivalence with intensional (relating) equivalence then replacement is achieved. Unfortunately, the property holds vacuously. 

\begin{fa}
	Let $A,B,\phi(A)\in\lMBCL$. If $\vdash_{\acMBCL} (A \to B) \land (B \to A)$ and $\vdash_{\acMBCL} \phi(A)$ then $\vdash_{\acMBCL} \phi(B)$.
\end{fa}

\begin{proof}
	Observe that the only properties of the relating relation holding for $\MBCL$ are $\scAi$, $\scAii$, $\scBo$, $\scBi$, $\scBii$, $\scCUN$. It is easy to check that there are no theorems of the form $(A \to B) \land (B \to A)$ in $\MBCL$. Therefore the premise of the fact is never satisfied and replacement is trivially valid.
\end{proof}

It is more interesting to consider some axiomatic extension of $\MBCL$ in which the premises of replacement are all realized. In that case, if no further constraints are imposed on $R$, replacement generally fails. Let us add to $\acMBCL$ the axioms $(A \land B) \to (B \land A)$ and $(A \land B) \to A$. This axiomatic extension gives a calculus which is complete w.r.t. the subclass of $\mathbf{F}_{B}$ of frames in which $\R_\oPos(A \land B, B \land A)$ and $\R_\oPos(A \land B, A)$ hold for every $\oPos\in\cPos$. Now define a single point model $\M$ whose frame belongs to this class and s.t. $\sim\R_\oPos(p \land q, q)$, for $p,q\in\Var$. Observe that $\R_\oPos(A \land B, A)$ allows conjunction simplification only on the first conjunct, it is order-sensitive. In this model we have $\M,\oPos\vDash ((p \land q) \to (q \land p)) \land ((q \land p) \to (p \land q))$, $\M,\oPos\vDash (p \land q) \to p$ but $\M,\oPos\nvDash (p \land q) \to q$. Therefore replacement generally fails. This negative result can be seen as an upside of the expansions of $\MBCL$ (and $\BCL$), which are hyperintensional systems capable of distinguishing between intensionally equivalent expressions.

\section{Motivation: syncategorematic negation and modalities}
\label{sec:3}

One of the first logical systems defined by relating semantics were presented by Richard Epstein in \cite{epstein1}. By means of the considered relating implication, Epstein captured a content-based conditional, where the content relationship was understood as overlapping of the subject matter of sentences.\footnote{It is also worth noting that in the paper of G\"ardenfors \cite{PG1978} there was an attempt to give postulates for relevance of two sentences in terms of a binary relation $\R$ which resembles pretty much the approach proposed by Epstein in \cite{epstein1}.}  The idea aroused the interest of other logicians (like D. Walton, D. Lewis and  R. Goldblatt), who noticed that Epstein's implication could be used to express the content relationship understood as overlapping of sentence content (see \cite[p. 113]{epstein1.5}; \cite[pp. 156-158]{epstein1}; \cite[pp. 68-70]{epstein2}). 

The logics he considered were named relatedness logics: the logic $\ES$ and the logic $\ER$, which is a sublogic of $\ES$.\footnote{The logic $\ER$ should not be confused with Anderson and Belnap's relevant logic {\bf R}.} The monograph \cite{epstein2} contains the most significant results on the logics defined by Epstein, for instance, the sound and complete axiomatization (cf. \cite{epstein1}).\footnote{More about the history of relating logic and Epstein's contribution can be found in \cite{klonowski2021b}.} 

The main postulate of Epstein's content-related sentences is that the subject matter of sentences is independent from the logical constants occurring in that sentence. This is explicitly given in the following passage:

\begin{quote}
	The logical connectives are syncategorematic: they are neutral with respect to subject matter (\cite[p. 66]{epstein2}).
\end{quote}

Considering a language with $\neg$, $\wedge$, $\rightarrow$ in the signature, the postulate is implemented by five assumptions (\cite[p. 66]{epstein2}):

\begin{description}
	\item $(\R1)$ $\R(A, \neg B) \Leftrightarrow \R(A, B)$
	\smallskip
	\item $(\R2)$ $\R(A, B \wedge C) \Leftrightarrow \R(A, B \rightarrow C)$
	\smallskip
	\item $(\R3)$ $\R(A, B) \Leftrightarrow \R(B, A)$
	\smallskip 
	\item $(\R4)$ $\R(A, A)$
	\smallskip
	\item $(\R5)$ $\R(A, B \wedge C) \Leftrightarrow \R(A, B) \text{ or } \R(A, C)$.
\end{description}

Since Boolean Connexive Logics are defined by relating semantics and additionally connexive logic is motivated by the idea of some special relationship between the sentences in the conditional, it seems natural to consider at least some of the $\R1$-$\R5$ conditions in the context of $\scAi$, $\scAii$, $\scBo$, $\scBi$, $\scBii$.

It is clear, however, that the combination of all these conditions leads to a contradiction, as stated by the following fact:

\begin{fa}\label{On inconsistent conditions} Let $\cRelsub \subseteq \cRel$.  If $\cRelsub$ contains exactly those relations that satisfiy one of the following sets of conditions:
	
	\begin{itemize}
		\item $\{ \scAi$,  $\R1$,  $\R4 \}$ 
		\item  $\{ \scAii$, $\R1$, $\R3$, $\R4 \}$
		\item $\{ \scBo$, $\scBi$,  $\R1 \}$ 
		\item $\{ \scBo$, $\scBii$,  $\R1 \}$,
	\end{itemize}
	
	\n $\cRelsub$ is the empty set. 
\end{fa}

The fact \ref{On inconsistent conditions} gives the exemplary sets of conditions that are inconsistent. Thus, the entire set of Epstein's content-related assumptions, together with the semantic assumptions of Boolean Connexive Logics, is inconsistent. And since it defines the empty set of relations $\cRelsub$, we consequently have an empty set of models  $\cModBCL{\cRelsub}$, but $\emptyset$ determines the  trivial logic (by \ref{ValidityBCL} or  by \ref{ValidityMBCL} in the modal context) - the set of all formulas.

However, instead of adding all these assumptions, or even some of them, we could consider incremental moves that bring the concept of connexivity closer to the relation between sentences as a kind of subject matter relation. Such a small move was made when in the semantics for Boolean Connexive Logic a closure condition for negation was imposed on the relating relation:

\begin{description}
	
	\item[$\scCUN$] $\R(\A, \B) \Rightarrow \R(\neg \A, \neg \B)$.
	\smallskip
\end{description}

\n Indeed, the $\scCUN$ condition states that adding the negation to two sentences does not invalidate the fact that the sentences are connected. Of course, this condition implies that we can add the negation $n$-times in this way by applying the condition $\scCUN$ $n$-times.

\begin{description}
	\item $\R( \A,  \B) \Rightarrow \R( \underbrace{\neg \dots \neg}_{n} \A, \underbrace{\neg \dots \neg}_{n} \B)$
\end{description}

But why would we add the negation the same number of times on both sides? Also, different numbers of additional negations could preserve the fact that the two sentences are connected. 

Moreover, the assumption that removing negations on both sides preserves the relationship between sentences seems to be another step towards Epstein's postulate. Thus, the opposite implication to $\scCUN$ deserves a separate study.

\begin{description}
	
	\item $\R(\neg \A, \neg \B) \Rightarrow \R(\A, \B)$.
	\smallskip
\end{description}

\n The implication can be generalized as described earlier, so that we can remove a different number of negations on both sides of the relation.

Putting the ideas together we can present the generalised closure under negation:

\begin{description}
	\item[$\scGCUN$] $\R(\underbrace{\neg \dots \neg}_{k} \A, \underbrace{\neg \dots \neg}_{l} \B) \Rightarrow \R( \underbrace{\neg \dots \neg}_{m} \A, \underbrace{\neg \dots \neg}_{n} \B)$
\end{description}

\n Condition $\scGCUN$ states that the negation is syncategorematic or partially syncategorematic, if we fix some numbers $k$, $l$, $m$, $n \in \mathbb{N}$. In the section  \ref{sec:4} we discuss the condition and propose an adequate axiomatization of Boolean Connexive Logics determined by $\scGCUN$ and its variants. 

The language considered by Epstein did not contain modalities, but the postulate of content independence from logical constants can also be extended to modalities. We consider $\MBCL$ which is defined on the set of formulas $\fMBCL$, containing $\Diamond$ and $\Box$.

In the paper \cite{JarmuzekMalinowski2019a}, it was proposed to consider such a possibility. Thus, having two sentences $p$ and $q$ that are connected: $\R(p, q)$, we can assume that the addition of any modality does not invalidate this state, so for example, also  $\R (\Diamond p, \Box q)$, etc. It is also possible to assume the opposite, i.e. that sentences with modalities that are connected do not cease to be related when the modality is removed. Thus, if $\R (\Diamond p, \Box q)$, then $\R(p, q)$. Finally, both assumptions can be made simultaneously. In order to present these three properties in general terms, the cited article provides a function that removes modalities from formulas. $\MBCL$ is significantly strengthened by imposing these properties on the classes of frames $\cRelMBCL$ and $\cRelMBCLCUN$. 

Undoubtedly, such an approach fits in with Epstein's postulate that the logical connectives are neutral with respect to subject matter, namely they are syncategorematic. 
In article \cite{JarmuzekMalinowski2019a}, we suggest the same about modalities: 

\begin{quote}
	Any modality can be treated – due to the Latin etymology of the word ``modality'' – as the way a modalized proposition holds. Term modality comes from the Latin world \textit{modus} which means a way; a way that something happens. 
	
	\,\,\,\,\, An option of non–treating modalities literally is to assume that modalities $\Box$, $\Diamond$ add nothing to the content of propositions modalized by them. For some people it may sound controversial, but \textit{modus} means a way, not a content.
\end{quote}

In the section \ref{sec:5} we formulate, discuss and finally axiomatize these Modal Boolean Connexive Logic systems, which are created by making assumptions about the syncategorematicity of $\Diamond$ and $\Box$.

\section{Closure under multiple negations}
\label{sec:4}

We have found that the property of closure under negation can be generalized to cover all the possible cases. Thus allowing for any number of negation symbols to appear in the antecedent and the consequent of the implication. Symbolically, a property of being closed under negation can be generalized for any number of negations in the following way. Let $\A, \B \in \fBCL$, $k,l,m,n \in \N$, $\R \in \cRel$, then:

\begin{description}
	\item[] $\R(\underbrace{\neg \dots \neg}_{k} \A, \underbrace{\neg \dots \neg}_{l} \B) \Rightarrow \R( \underbrace{\neg \dots \neg}_{m} \A, \underbrace{\neg \dots \neg}_{n} \B)$
\end{description}

\n Let us shorten down the notation by adding a superscript to negation symbol to express the number of occurrences. Thus, $\neg^{k} \A$ would denote the same expression as $\underbrace{\neg \dots \neg}_{k} \A$. This way, the presented property of $\R$ shall have the following form: 

\begin{description}
	\item[$\scGCUN$] $\R(\neg^{k} \A, \neg^{l} \B) \Rightarrow \R( \neg^{m} \A, \neg^{n} \B)$.
\end{description}

Note that $\scGCUN$ can be understood in several ways, e.g. we can fix some four numbers $k$, $l$, $m$, $n \in \N$ and only for them the given property is defined. In this case, with fixed numbers $k$, $l$, $m$, $n \in \N$, we will denote this property as $\scCUNk$.

Another way to understand this is to consider $k$, $l$, $m$, $n$ as variables and add different quantifier configurations as a prefix. Since we strive for generality, we will assume that the property is preceded by a universal quantification, so that the property is to hold for all natural numbers $k$, $l$, $m$, $n$. We will denote this variant of the property by the abbreviation $\scGCUN$.

The problem, however, is that in the general sense of the condition $\scGCUN$, the only relation class that satisfies this condition and the consequential conditions together is the empty class. These conditions are inconsistent. We have the fact:

\begin{fa}\label{gcun condition is inconsistent with connexive conditions} Let $\cRelsub \subseteq \cRel$.  If $\cRelsub$ contains exactly those relations that satisfy $\scGCUN$, $\scBo$,  and $\scBi$ or $\scBii$, $\cRelsub$ is the empty set. 
	
\end{fa}

\begin{proof}
	Let us make the assumptions. For the proof that imposing $\scBi$ on $\cRelsub$ makes the set empty,  in the schema $\scGCUN$ we take $k = 0$, $l = 1$, $m = 0$, $n = 2$,  for $A$:  $(A \rightarrow B)$, and for $B$: $(A \rightarrow \neg B)$. So, we have: 
	
	\[ \R((A \rightarrow B), \neg (A \rightarrow \neg B)) \Rightarrow \R((A \rightarrow B), \neg \neg (A \rightarrow \neg B)).\]
	
	\n But the antecedent of the implication is $\scBi$, so we conclude that $\R((A \rightarrow B), \neg \neg (A \rightarrow \neg B))$. Now we make use of $\scBo$. Instead of $A$ we take $(A \rightarrow B)$, and for $B$: $(A \rightarrow \neg B)$. So, we have: 
	
	\[ \R((A \rightarrow B), \neg (A \rightarrow \neg B)) \Rightarrow \, \sim \R((A \rightarrow B), \neg \neg (A \rightarrow \neg B)).\]
	
	\n Again we detach $\scBi$, and we get $\sim \R((A \rightarrow B), \neg \neg (A \rightarrow \neg B)).$ which is a contradiction. 
	
	For the case of $\scBii$, we assume also $k = 0$, $l = 1$, $m = 0$, $n = 2$ in $\scGCUN$, but substitute  $\scGCUN$ and  $\scBo$ such that we can detach two times $\scBii$. This again leads to contradiction. 
\end{proof}

Thus $\scGCUN$ with $\scAi$, $\scAii$, $\scBo$, $\scBi$, and $\scBii$ result in the  trivial logic (by \ref{ValidityBCL} or  by \ref{ValidityMBCL} in the modal context). The only reasonable option is to use the property $\scCUNk$ for some quadruples that do not lead to contradiction. 
Among a countable set of such quadruples, at least some do not lead to a contradiction. An example is the condition $\scCUN$, which is based on the quadruple: 0, 0, 1, 1. On the other hand, in the proof of fact \ref{gcun condition is inconsistent with connexive conditions}, the quadruple 0, 1, 0, 2 resulted in a contradiction. In the following, we will introduce the axiomatization method for  $\scCUNk$.

For further investigations, let us assume some four natural numbers: $k$, $l$, $m$, $n$. Before we move on to the next axioms, let's consider what would happen if $k = l = m = n$. Another fact tells about this:

\begin{fa}\label{scunk needs no axioms if k=l=m=n} Let $\cRelsub \subseteq \cRel$. Let $k=l=m=n$. If $\cRelsub$ contains exactly those relations that satisfy $\scCUNk$, $\cRelsub = \cRel$.  
	
\end{fa}

\begin{proof} If $k=l=m=n$, then $\scCUNk$ takes the form:
	
	\[\R(\neg^{k} \A, \neg^{k} \B) \Rightarrow \R( \neg^{k} \A, \neg^{k} \B),\]
	
	\n which is an instance of classical tautology, so any binary relation satisfies it, particularly any relation in $\cRel$.  
	
\end{proof}

So, the case when $k=l=m=n$ does not define any new class of connexive relations or  frames.  But what, if the four numbers are not equal? 

So far we are able to provide complete syntactic calculi for classes of models closed under $\scCUNk$ for quadruples of values $k,l,m,n$ satisfying specific conditions. We generalize the strategy developed by Klonowski \cite{Klonowski2}. As already seen, among the axioms of $\BCLCUN$ the ones which allow to express the property of closure under negation (cun) inside the syntax are the pair:

\begin{description}
	\item[$\aCUNi$] $(\A \rightarrow \B) \supset ((\neg \A \rightarrow \neg \B) \vee (\neg \A \wedge \B))$
	\smallskip
	\item[$\aCUNii$] $(\A \rightarrow \B) \supset (\neg \neg \A \rightarrow \neg \neg \B)$
\end{description}

Obviously $\aCUNi$ is an instance of $\scCUNk$ when we take the quadruple of values: $0,0,1,1$. For arbitrary quadruples, it is plausible to expect the schema to generalize to:

\begin{description}
	\item[$\aGCUN$] $(\neg^k \A \rightarrow \neg^l \B) \supset ((\neg^m \A \rightarrow \neg^n \B) \vee (\neg^m \A \wedge \neg^{n+1} \B))$
\end{description}

The situation is less straightforward for a generalization of $\aCUNii$. This axiom expresses the result of a double application of the property $\scCUNk$ Moreover we want said application to preserve the truth-value of the transformed formulas, since they are connected by the connexive arrow, which in order to be true requires also the truth of its respective material implication. Therefore the desired axiom must express an iteration of $\scCUNk$ which preserves the truth-value of each transformed formula.

Let us fix $k,l,m,n \in \mathbb{N}$, assuming that $k \leq m, l \leq n$. We start from a graphical consideration. Property $\scCUNk$ states that when $\neg^k \A$ and $\neg^l \B$ are related so are $\neg^m \A$ and $\neg^n \B$. Let us focus on $A$. The axiom we are looking for should transform $\neg^k \A$ into $\neg^m \A$. Now $\neg^m \A$ is just a graphical variant of $\neg^k \neg^{m-k} \A$, therefore this formula can be transformed again into $\neg^m \neg^{m-k} \A$, which is the same as $\neg^{2m-k} \A$. The same reasoning applies to $\neg^l B$, which after two tranformations becomes $\neg^{2n-l} \B$. We want the iterated tranformation to preseve truth-values and since negation behaves classically in our framework, the only way for $\neg^k \A$ and $\neg^{2m-k} \A$ to have the same truth-values is for $2m-k$ to be even, thefore $k$ has to be even. By the same reasoning $l$ has to be even as well. 

We conclude that accepting the constraints that $k \leq m, l \leq n$, with $k,l$ even, the following axiom is a good candidate for a generalization of $\aCUNii$:

\begin{description}
	\item[$\aGCUNii$] $(\neg^k \A \rightarrow \neg^l \B) \supset (\neg^{2m-k} \A \rightarrow \neg^{2n-l} \B)$
\end{description}

Before we move to the completeness proof, we need to fix some notation. Let us fix $k,l,m,n \in \mathbb{N}$, with $k \leq m, l \leq n$, and $k,l$ even. Let $\cRelBCLGCUN$ denotes the class of all relations that satisfy the $\BCLCUN$ conditions and $\scCUNk$. The class of all models built over those relations will be denoted by $\cModBCL{\cRelBCLGCUN}$.

The axiom system $\acBCLGCUN$ is the Hilbert-style calculus obtained from $\aCPL$, $\aAi$, $\aAii$, $\aBi$, $\aBii$, $\aCUNi$, $\aCUNii$, $\aImp$, $\acDS$, by the addition of schemata $\aGCUN$, $\aGCUNii$ for the same quadruple of values fixed above. 

\begin{tw}[Soundness Theorem for $\acBCLGCUN$]\label{SoundnessBCL}
	For any $\Gamma \subseteq \fBCL, \A \in \fBCL$, $\Gamma \vdash_{\acBCLGCUN} \A$ $\Rightarrow$ $\Gamma \vDash_{\cRelBCLGCUN} \A$.
\end{tw}

\begin{proof} Let us consider all axiom schemata and the rule contained in $\acBCLGCUN$. 
	
	\textit{Classical Propositional Logic Axioms}. According to the definition \ref{InterpretationBCL}, we got that $\CPL \subseteq \cRelBCL$. Thus, any axiom of $\CPL$ is valid in $\cRelBCL$.
	
	\textit{Disjunctive Syllogism Rule}. Let us assume that $\vDash \B$ and $\vDash \B \supset \C$ for some $\B, \C \in \fBCL$. Then according to definition \ref{InterpretationBCL}, we have that $\eVal(\B) = 1$ and ($\eVal(\B) = 0$ or $\eVal(\C) = 1$). Hence, $\eVal(\C) = 1$, from which it follows that $\vDash \C$.
	
	\textit{$\aAi$}. Let $\M = \langle \eVal, \R \rangle \in \cModBCL{\cRelBCLGCUN}$. Then $\sim \R(\A, \neg \A)$ by $\scAi$. Hence $\M \nvDash \A \to \neg \A$, that is $\M \vDash \neg (\A \to \neg \A)$.
	
	\textit{$\aAii$}. Similar to $\aAi$, using property $\scAii$.
	
	\textit{$\aBi$}. Let us assume that $\M \nvDash (\A \rightarrow \B) \rightarrow \neg (\A \rightarrow \neg \B)$ for some $\M \in \cModBCL{\cRelBCLGCUN}$. According  to definition \ref{InterpretationBCL}, (1) $\M \nvDash (\A \rightarrow \B) \supset \neg (\A \rightarrow \neg \B)$ or (2) $\sim \R((\A \rightarrow \B), \neg (\A \rightarrow \neg \B))$. (2) cannot hold since by lemma $\ref{ClassLemmaBCL}$ $\scBo$ holds. If (1), then $\M \vDash (\A \rightarrow \B)$ and $\M \nvDash \neg (\A \rightarrow \neg \B)$, which implies both $\R(\A, \B)$ and $\R(\A, \neg \B)$, contradicting $\scBi$.
	
	\textit{$\aBii$}. Similar to $\aBi$, using property $\scBii$.
	
	\textit{$\aImp$}. Let us assume that $\vDash \A \rightarrow \B$. Then, due to definition \ref{InterpretationBCL}, $\vDash \neg \A \vee \B$ and $\R(\A, \B)$. According to the abbreviation for material implication, we obtain  $\vDash \A \supset \B$.
	
	\textit{$\aGCUN$}. Let us assume that $\M \nvDash (\neg^{k} \A \rightarrow \neg^{l} \B) \supset ((\neg^{m} \A \rightarrow \neg^{n} \B) \vee (\neg^{m} \A \wedge \neg^{n + 1} \B))$, therefore $\M \vDash (\neg^{k} \A \rightarrow \neg^{l} \B)$ and $\M \nvDash (\neg^{m} \A \rightarrow \neg^{n} \B)$ and $\M \nvDash (\neg^{m} \A \wedge \neg^{n + 1} \B)$. From $\M \nvDash (\neg^{m} \A \rightarrow \neg^{n} \B)$ it follows that either $\M \nvDash (\neg^{m} \A \supset \neg^{n} \B)$ or $\sim \R (\neg^{m} \A, \neg^{n} \B)$. The latter cannot hold, since $\M \vDash (\neg^{k} \A \rightarrow \neg^{l} \B)$ implies $\R (\neg^{m} \A, \neg^{n} \B)$. The former holds iff $\M \vDash \neg^{m} \A$ and $\M \nvDash \neg^{n} \B$, hence $\M \vDash \neg^{n+1} \B$, obtaining a contradiction.
	
	\textit{$\aCUNi$}. Similar to $\aGCUN$.
	
	\textit{$\aGCUNii$}. Let us assume that $\M \vDash \neg^{k} \A \rightarrow \neg^{l} \B$, therefore $\M \vDash (\neg^{k} \A \supset \neg^{l} \B)$ and $\R (\neg^{k} \A, \neg^{l} \B)$. Since $2m-k, 2n-l$, also $\M \vDash (\neg^{k} \A \supset \neg^{l} \B)$. From $\R (\neg^{k} \A, \neg^{l} \B)$ a double application of $\scCUNk$ yelds $\R (\neg^{2m-k} \A, \neg^{2n-l} \B)$, therefore $\M \vDash (\neg^{2m-k} \A \rightarrow \neg^{2n-l} \B)$.
	
	\textit{$\aCUNii$}. Similar to $\aGCUNii$.
\end{proof}

We are going to employ the standard notions of consistent and maximal consistent sets of formulas. Let us denote $\fGen$ as the formulas of some language $\lGen$ and $\acGen$ be some axiom system expressed in the same language. 

\begin{df}[Consistent Set of Formulas]
	Let $\Gamma \subseteq \fBCL$. Then, $\Gamma$ is a $\acGen$-consistent set of formulas iff for some $\A \in \fBCL$, $\Gamma \nvdash_{\acGen} \A$. Otherwise, it is called a $\acGen$-inconsistent set of formulas. 
\end{df}

\begin{df}[Maximal Consistent Set of Formulas]
	Let $\Gamma \subseteq \fGen$. We call $\Gamma$ a maximal $\acGen$-consistent iff both conditions are satisfied: 
	\begin{description}
		\item[(a)] $\Gamma$ is $\acGen$-consistent
		\smallskip
		\item[(b)] for any $\Delta \subseteq \fGen$, $\Gamma \subset \Delta \Rightarrow \Delta$ is $\acGen$-inconsistent.
	\end{description}
\end{df}

We will denote the class of all maximal $\acGen$-consistent sets of formulas as $\cGenMax$. Analogous notation will be used for $\BCL$ and $\MBCL$. 

We can now proceed and provide  a proof for completeness using the modified notion of canonical model for $\BCL$ defined by Klonowski in (see \cite{Klonowski2}, p. 529). We are going to build two canonical models, such that they make true the same formulas, but only the second of them is guaranteed to belong to the desired class of models.

\begin{df}[First Canonical Model for $\acBCLGCUN$]\label{CanonicalModelBCL}
	
	Let $\Gamma \in \cBCLGCUNMax$, the first canonical model generated by $\Gamma$ (in short $\Gamma$-model) is a pair $\langle \eVal, \R \rangle$ such that, for any $\A, \B \in \fBCL$:
	\begin{description}
		\item[(a)] $\eVal(p) = 1 \Leftrightarrow p \in \Gamma$, for $p \in \Var$
		\smallskip
		\item[(b)] $\R(\A,\B) \Leftrightarrow \A \rightarrow \B \in \Gamma$
	\end{description}
\end{df}

\begin{lm}\label{CanonicalLemmaBCL}
	Let $\Gamma \in \cBCLGCUNMax$ and $\M$ be a $\Gamma$-model. Then, for all $\A \in \fBCL$, $\M \vDash \A \Leftrightarrow \A \in \Gamma$.
\end{lm}

\begin{proof}
	Assume all the hypotheses. We will use induction over the complexity of formulas.
	
	\textit{Base case}. Let $\A \in \Var$. Then $\M \vDash \A \Leftrightarrow \eVal(\A) = 1$. By definition \ref{CanonicalModelBCL}, we obtain that it is equivalent to $\A \in \Gamma$.
	
	\textit{Inductive hypothesis}. Let $\oNum \in \N$. Let us suppose that the theorem is true for all the formulas whose complexity is lesser or equal than $\oNum$.
	
	\textit{Inductive step}. Let us assume that the formula $\A$ has complexity equal to $\oNum + 1$. Thus, the possible cases are $\A = \neg \B$, $\A = \B \wedge \C$, $\A = \B \vee \C$ and $\A = \B \rightarrow \C$ for some $\B, \C \in \fBCL$. Due to the basic properties of maximal consistent sets and definition $\ref{CanonicalModelBCL}$ we can easily prove the theorem for complex formulas constructed using classical connectives. 
	
	\textit{Let $\A = \B \rightarrow \C$.} `\textit{$\Rightarrow$}'. Suppose that $\M \vDash \B \rightarrow \C$. Then, by the definition \ref{InterpretationBCL}, we obtain ($\M \nvDash \B$ or $\M \vDash \C$) and $\R(\B,\C)$. By the definition \ref{CanonicalModelBCL} and from the fact that $\R(\B,\C)$, we obtain $\B \rightarrow \C \in \Gamma$.
	
	`\textit{$\Leftarrow$}'. Suppose that $\B \rightarrow \C \in \Gamma$. Thus, by $\aImp$ and the usual properties of a maximal consistent set, we obtain $\B \notin \Gamma$ or $\C \in \Gamma$. Hence, by the inductive hypothesis, $\M \nvDash \B$ or $\M \vDash \C$. Moreover by the definition \ref{CanonicalModelBCL}, $\B \rightarrow \C \in \Gamma$ implies $\R(\B, \C)$. By definition \ref{InterpretationBCL}, we conclude $\M \vDash \B \rightarrow \C$.
\end{proof}

It is essential to the completeness proof for the canonical model to belong to the class $\cModBCL{\cRelBCLGCUN}$. It is for this reason that we introduce a second canonical model\footnote{Condition (c) of definition \ref{CanonicalModelBCLsecond} can be redundant for certain choices of $k,l,m,n$, but we consider the general case here.}.

\begin{df}[Second Canonical Model for $\acBCLGCUN$]\label{CanonicalModelBCLsecond}
	
	Let $\Gamma \in \cBCLGCUNMax$ and $\M = \langle \eVal_\M, \R \rangle$ be a $\Gamma$-model. Let $\R^\neg \subseteq \fBCL \times \fBCL$ be the least relation $\rho$ such that:
	\begin{description}
		\item[(a)] $\M \vDash \A \rightarrow \B \Rightarrow \rho (\A, \B)$
		\smallskip
		\item[(b)] $\M \vDash \neg^{k} \A \rightarrow \neg^{l} \B$ and $\M \vDash \neg^{m} \A \wedge \neg^{n + 1} \B \Rightarrow \rho (\neg^{m} \A, \neg^{n} \B)$
		\smallskip
		\item[(c)] $\M \vDash \A \rightarrow \B$ and $\M \vDash \neg \A \wedge \B \Rightarrow \rho (\neg\A, \neg\B)$
	\end{description}
	
	\noindent The second canonical model generated by $\Gamma$ (in short $\Gamma^\neg$-model) is the pair $\M^\neg = \langle \eVal_\M, \R^\neg \rangle$.
\end{df}

\begin{lm}\label{can_mod1-2LemmaBCL}
	For $\Gamma \in \cBCLGCUNMax$, let $\M = \langle \eVal, \R \rangle$ be a $\Gamma$-model and $\M^\neg = \langle \eVal, \R^\neg \rangle$ be its $\Gamma^\neg$-model. Then for all $\A \in \fBCL, \M \vDash \A \Leftrightarrow \M^\neg \vDash \A$.
\end{lm}

\begin{proof}
	We proceed by induction over the complexity of formulas. Since the two models share the same valuation $\eVal$, the only non-trivial case is the one concerning the connexive arrow. 
	
	\textit{Let $\A = \B \rightarrow \C$.} `\textit{$\Rightarrow$}'. Suppose that $\M \vDash \B \rightarrow \C$. Then, by definition \ref{InterpretationBCL}, we obtain ($\M \nvDash \B$ or $\M \vDash \C$) and $\R(\B,\C)$. By induction hypothesis, $\M^\neg \nvDash \B$ or $\M^\neg \vDash \C$. By definition \ref{CanonicalModelBCLsecond}, $\M \vDash \B \rightarrow \C$ implies $\R^\neg(\B,\C)$. We obtain $\M^\neg \vdash \B \rightarrow \C$.
	
	`\textit{$\Leftarrow$}'. Suppose that $\M^\neg \vDash \B \rightarrow \C$. Then, by definition \ref{InterpretationBCL}, we obtain ($\M^\neg \nvDash \B$ or $\M^\neg \vDash \C$) and $\R^\neg (\B,\C)$. Now, by definition \ref{CanonicalModelBCLsecond} $\R^\neg (\B,\C)$ can hold only in two cases: either (a) $\M \vDash \B \rightarrow \C$ or (b) $\B = \neg^m \B_1, \C = \neg^n \C_1, \M \vDash \neg^k \B_1 \to \neg^l \C_1$ and $\M \vDash \neg^m \B_1 \land \neg^{n+1} \C_1$. If (a) holds, we are done. If (b) is the case, since $\M^\neg \nvDash \B$ or $\M^\neg \vDash \C$, it is the same as $\M^\neg \nvDash \neg^k \B_1$ or $\M^\neg \vDash \neg^l \C_1$, and by induction hypothesis $\M \nvDash \neg^k \B_1$ or $\M \vDash \neg^l \C_1$, which contradicts $\M \vDash \neg^m \B_1 \land \neg^{n+1} \C_1$. Therefore (a) must hold.
\end{proof}

\begin{lm}\label{ClassLemmaBCL}
	For $\Gamma \in \cBCLGCUNMax$, let  $\M^\neg = \langle \eVal, \R^\neg \rangle$ be a $\Gamma^\neg$-model. Then $\M^\neg \in \cModBCL{\cRelBCLGCUN}$.
\end{lm}

\begin{proof}
	Let $\M = \langle \eVal, \R \rangle$ be the $\Gamma$-model from which $\M^\neg$ is built. Assume $\R^\neg(\neg^{k} \A, \neg^{l} \B)$. By definition \ref{CanonicalModelBCLsecond}, this holds only in two cases: either (a) $\M \vDash \neg^k \A \rightarrow \neg^l \B$ or (b) $\neg^k \A = \neg^m \A_1, \neg^l \B = \neg^n \B_1, \M \vDash \neg^k \A_1 \to \neg^l \B_1$ and $\M \vDash \neg^m \A_1 \land \neg^{n+1} \B_1$. 
	
	If (a) holds, by properties of maximal consistent sets, $\M \vDash \neg^k \A \rightarrow \neg^l \B$ implies $\M \vDash (\neg^m \A \rightarrow \neg^n \B) \lor (\neg^m \A \land \neg^{n+1} \B)$ by $\aGCUN$. Therefore by the same properties, $\M \vDash \neg^m \A \rightarrow \neg^n \B$ or $\M \vDash \neg^m \A \land \neg^{n+1} \B$. In case $\M \vDash \neg^m \A \rightarrow \neg^n \B$, by definition \ref{CanonicalModelBCLsecond}, this implies $\R^\neg (\neg^m \A, \neg^n \B)$. If $\M \vDash \neg^m \A \land \neg^{n+1} \B$, by the same definition and $\M \vDash \neg^k \A \rightarrow \neg^l \B$ it follows $\R^\neg (\neg^m \A, \neg^n \B)$.
	
	If (b) is the case, $\M \vDash \neg^k \A_1 \to \neg^l \B_1$ and $\aGCUNii$ imply $\M \vDash \neg^{2m-k} \A_1 \to \neg^{2n-l} \B_1$. Notice that $\neg^{2m-k} \A_1 = \neg^{m-k} \neg^m \A_1 = \neg^{m-k} \neg^k \A = \neg^m \A$ and $\neg^{2n-l} \B_1 = \neg^{n-l} \neg^n \B_1 = \neg^{n-l} \neg^l \B = \neg^n \B$. Hence $\M \vDash \neg^{m} \A \to \neg^{n} \B$, which by definition \ref{CanonicalModelBCLsecond} allow us to conclude $\R^\neg (\neg^m \A, \neg^n \B)$.
	
	For all the other properties, the proof follows \cite{Klonowski2}, proposition 4.11. In particular, that $\M^\neg$ is closed under negation per $\scCUN$ is obtained reasoning as above, using clause (c) of definition \ref{CanonicalModelBCLsecond} with theorems $\aCUNi,\aCUNii$.
\end{proof}

The second part of the last theorem, which concerns all the connexive properties, is working in particular thanks to the closure $\scCUN$. This motivates the decision to include $\aCUNi, \aCUNii$ in the axiomatization. There is more left to explore here. At least, two directions can be considered. We can study classes of models in which the connexive laws stills hold but such that $\scCUN$ is no longer guaranteed. In that sense the focus is on the interconnection between the properties characterizing connexivity and the iteration of negations, to see how far this can go while preserving both connexive laws and consistency. A diverging path is the one leading to the more general framework of relating semantics. In that sense, the main interest would shift over the way of characterizing classes of models closed under multiple negations, leaving aside connexivity and studying all the possible choices of values for $\scCUNk$. Both suggetions will be left for future works.

Returning to the main proof, we have shown that $\M$ and $\M^\neg$ behaves semantically in the same way. Moreover $\M^\neg \in \cModBCL\cRelBCLGCUN$. Therefore we can work inside $\M$ and be sure that there exists another model in the desired class which satisfies precisely the same formulas.

\begin{tw}[Completeness Theorem for $\acBCLGCUN$]
	
	For any $\Gamma \subseteq \fBCL, \A \in \fBCL$, $\Gamma \vDash_{\cRelBCLGCUN} \A$ $\Rightarrow$ $\Gamma \vdash_{\acBCLGCUN} \A$.
\end{tw}

\begin{proof}
	By contraposition, let us assume that $\Gamma \nvdash_{\acBCLGCUN} \A$. Then $\Gamma \cup \{ \neg \A \}$ is consistent. According to the Lindebaum's lemma, there exists a maximal consistent extension of this set $\Delta$. According to lemma \ref{CanonicalLemmaBCL}, there exists a first canonical model $\M$ for which $\M \vDash P$ iff $P \in \Delta$, Thus, $\M \nvDash \A$ and, for all $\B \in \Gamma, \M \vDash \B$. By \ref{can_mod1-2LemmaBCL}, also the second canonical model $\M^\neg$ based on $\M$ is such that $\M^\neg \nvDash \A$ and, for all $\B \in \Gamma, \M^\neg \vDash \B$. Finally $\M^\neg \in \cModBCL{\cRelBCLGCUN}$ by \ref{ClassLemmaBCL}. Therefore $\Gamma \nvDash_{\cRelBCLGCUN} \A$.
\end{proof}

Analogous results can be achieved for $\MBCL$ closed under multiple negations. Due to the fact that the semantics is more complicated, there is a need of introducing a new notion of a canonical model. As well as in the previous case, Klonowski defined the notion of a canonical model for $\MBCL$ in (see \cite{Klonowski2}, p. 532), which we will modify. 

Symbolically, the property of being closed under negation for $\MBCL$ looks very similar to the corresponding property for $\BCL$. The change is within the subscript of $\R$, which is related to some possible world $\oPos \in \cPos$. Let $\A, \B \in \fMBCL$, $k,l,m,n \in \N$, $\oPos \in \cPos$, and $\R_{\oPos} \in \cRel$, then, with a slight abuse of notation: 

\begin{description}
	\item[$\scGCUN$] $\R_{\oPos}(\neg^{k} \A, \neg^{l} \B) \Rightarrow \R_{\oPos}(\neg^{m} \A, \neg^{n} \B)$
\end{description}

Let us denote the class of all the combined frames whose relations satisfy the $\MBCLCUN$ conditions and $\scGCUN$ as $\cRelMBCLGCUN$. The class of all models built over those  frames will be denoted by $\cModMBCL{\cRelMBCLGCUN}$. The property can be expressed in the syntax by the schemata already introduced, $\aGCUN, \aGCUNii$. The axiom system $\acMBCLGCUN$ is the Hilbert-style calculus obtained from $\acBCLGCUN$ adding the modal schemata $\aDual$, $\aKm$, and the rule $\anec$.

\begin{tw}[Soundness Theorem for $\acMBCLGCUN$]\label{SoundnessMBCLGCUN}
	For any $\Gamma \subseteq \fMBCL, \A \in \fMBCL$, $\Gamma \vdash_{\acMBCLGCUN} \A \Rightarrow$  $\Gamma \vDash_{\cRelMBCLGCUN} \A$.
\end{tw}

\begin{proof} All the cases considered in theorem $\ref{SoundnessBCL}$ still hold. Let us consider the remaining ones.
	
	\textit{Neccessitation Rule}. Let us assume that $\vDash \A$. Then, $\M, \oPos \vDash \A$ for any $\M \in \cModMBCL{\cRelMBCLGCUN}$ and $\oPos \in \cPos$. In particular, for all $\oPos_{1} \in \cPos$ such that $\rAvail(\oPos, \oPos_{1}), \M, \oPos_{1} \vDash \A$. Therefore $\vDash \Box \A$.
	
	\textit{$\aDual$}. Let $\M, \oPos \vDash \Diamond \A$. Then, according to definition \ref{InterpretationMBCL}, this holds iff there exists $\oPos_{1} \in \cPos$ such that $\rAvail(\oPos, \oPos_{1})$ and $\M, \oPos_{1} \vDash \A$. Which is equivalent to the fact that it is not the case that for every $\oPos_{2} \in \cPos$ such that $\rAvail(\oPos, \oPos_{2})$ we have $\M, \oPos_{1} \nvDash \A$. By the same definition \ref{InterpretationMBCL}, this amounts to $\M, \oPos \vDash \neg \Box \neg \A$.
	
	\textit{$\aK$}. Let $\M, \oPos \vDash \Box (\A \supset \B)$ and $\M, \oPos \vDash \Box \A$. Then for all $\oPos_{1} \in \cPos$ such that $\rAvail(\oPos, \oPos_{1}), \M, \oPos_{1} \vDash \A \supset \B$ and $\M, \oPos_{1} \vDash \A$. Hence $\M, \oPos_1 \vDash \B$, therefore $\M, \oPos \vDash \Box \B$.
\end{proof}

For completeness, we adapt the same strategy employed for the non-modal case, building two canonical models. 

\begin{df}[First Canonical Model for $\acMBCLGCUN$]\label{CanonicalModelMBCLGCUN}
	
	The first canonical model for $\acMBCL$ is a quadruple $\langle \cPos, \rAvail, \{\R_{\oPos}\}_{\oPos \in \cPos}, \eVal \rangle$ such that, for any $\A, \B \in \fMBCL$:
	\begin{description}
		\item[(a)] $\cPos = \cMBCLGCUNMax$
		\smallskip
		\item[(b)] $\rAvail(\oPos, \oPos_{1}) \Leftrightarrow$  for all $\A \in \fMBCL, (\square \A \in \oPos \Rightarrow \A \in \oPos_{1})$
		\smallskip
		\item[(c)] for all $\oPos \in \cPos, p \in \Var$:
		$$\eVal(\oPos, p) =
		\begin{cases} 
			1 & $ if $p \in \oPos \\
			0 & $ if $p \notin \oPos
		\end{cases}$$
		\smallskip
		\item[(d)] $\R_{\oPos}(\A, \B) \Leftrightarrow \A \rightarrow \B \in \oPos$
	\end{description}
\end{df}

\begin{df}[Second Canonical Model for $\acMBCLGCUN$]\label{CanonicalModelMBCLsecond}
	
	Let $\M = \langle \cPos, \rAvail, \{\R_{\oPos}\}_{\oPos \in \cPos}$, $\eVal \rangle$ be the first canonical model for $\acMBCL$. For $\oPos \in \cPos$, let $\R_\oPos^\neg \subseteq \fMBCL \times \fMBCL$ be the least relation $\rho$ such that:
	\begin{description}
		\item[(a)] $\M, \oPos \vDash \A \rightarrow \B \Rightarrow \rho (\A, \B)$
		\smallskip
		\item[(b)] $\M, \oPos \vDash \neg^{k} \A \rightarrow \neg^{l} \B$ and $\M, \oPos \vDash \neg^{m} \A \wedge \neg^{n + 1} \B \Rightarrow \rho (\neg^{m} \A, \neg^{n} \B)$
		\smallskip
		\item[(c)] $\M, \oPos \vDash \A \rightarrow \B$ and $\M, \oPos \vDash \neg \A \wedge \B \Rightarrow \rho (\neg\A, \neg\B)$
	\end{description}
	
	\noindent The second canonical model for $\acMBCLGCUN$ is the tuple $\M^\neg = \langle \cPos, \rAvail$, $\{\R^\neg_{\oPos}\}_{\oPos \in \cPos}$, $\eVal \rangle$.
\end{df}

The construction follows precisely the steps for the non-modal case, adapting it to the new modal setting. Now the relation $\R$ is extended to a family in which each relation is indexed to a possible world. From the first canonical model, we build a second one by substituting each relation $\R_\oPos$ with its counterpart $\R_\oPos^\neg$ closed under $\scGCUN$.

As expected, all the facts previously proved hold in their modal version as well. Since the proofs are, with the appropriate notation, identical to the ones already given, we simply list the results.

\begin{lm}\label{CanonicalLemmaMBCLGCUN}
	Let $\M$ be the first canonical model for $\acMBCLGCUN$. Then, for all $\oPos \in \cPos_\M, \A \in \fMBCL$, $\M, \oPos \vDash \A \Leftrightarrow \A \in \oPos$.
\end{lm}

\begin{lm}
	Let $\M$ be the first canonical model for $\acMBCLGCUN$, and $\M^\neg$ be the second one. Then for all $\oPos \in \cPos, \M, \oPos \vDash A \Leftrightarrow \M^\neg, \oPos \vDash A$.
\end{lm}

\begin{lm}
	Let $\M^\neg$ be the second canonical model for $\acMBCLGCUN$. Then $\M^\neg \in \cModBCL{\cRelMBCLGCUN}$.
\end{lm}

\begin{tw}[Completeness Theorem for $\acMBCLGCUN$]
	For any $\Gamma \subseteq \fMBCL$, $\A \in \fMBCL$, $\Gamma \vDash_{\cRelMBCLGCUN} \A$ $\Rightarrow$ $\Gamma \vdash_{\acMBCLGCUN} \A$.
\end{tw}

\section{Closure under demodalization}
\label{sec:5}

Besides the fact that the relating relation in $\MBCL$ could be closed under multiple negations, there is also a possibility of closing it under a demodalization function. This function simplifies a given expression by removing all modal symbols. By closing the relation $\R$ under the mentioned function, we obtain a connection between expressions containing modal symbols and expressions resulting from removal of those symbols.

We can distinguish three different types of connections between those two pairs of formulas: (1) if two formulas containing modal symbols are in relation $\R$, then the formulas resulting from them by removing modal symbols are also in relation $\R$, (2) the implication the other way round, (3) the equivalence.\footnote{Since it will not be relevant to the current discussion, we leave aside the mixed cases in which we start from a pair containing both a modal formula and a demodalized one and obtain a pair of a demodalized and a modal ones. Using the notation that is going to explained in the following, these are the cases:
	
	\begin{description}
		\item[($\mathsf{DemM1}$)] $\RModal(\A, \fDem(\B)) \Rightarrow \RModal(\fDem(\A), \B)$
		\smallskip
		\item[($\mathsf{DemM2})$] $\RModal (\fDem(\A), \B) \Rightarrow \RModal(\A, \fDem(\B))$
\end{description}}

In this section, we will provide semantic conditions matching described types of connections and present axiom schemata for each of them. Having this, we will use the notion of a canonical model for $\MBCL$ used in the previous section, to prove completeness and soundness for those three axiom systems. 

Let us start with the definition of a demodalization function. The function was introduced by Malinowski and Jarmu\.zek in (see \cite{JarmuzekMalinowski2019}, p. 227). 

\begin{df}
	
	A function $\fDem: \fMBCL \longrightarrow \fMBCL$ will be called demodalization if it satisfies following conditions:
	\begin{description}
		\item[(a)] $\fDem(\A) = \A$ if $\A \in \Var$
		\smallskip
		\item[(b)] $\fDem(\neg \A) = \neg \fDem(\A)$
		\smallskip
		\item[(c)] $\fDem(\A \star \B) = \fDem(\A) \star \fDem(\B)$, where $\star \in \lbrace \wedge, \vee, \rightarrow \rbrace$
		\smallskip
		\item[(d)] $\fDem(\ast \A) = \A$, where $\ast \in \lbrace \Diamond, \Box \rbrace$.
	\end{description}
\end{df}

Using defined notions, we can introduce conditions for closure under demodalization. Let $\M = \langle \cPos, \rAvail, \eVal, \cRModal \rangle$, and $\oPos \in \cPos$, then:

\begin{description}
	\item[$\scDemR$] $\RModal(\A, \B) \Rightarrow \RModal( \fDem(\A), \fDem(\B))$
	\smallskip
	\item[$\scDemL$] $\RModal(\fDem(\A), \fDem(\B)) \Rightarrow \RModal(\A, \B)$
	\smallskip
	\item[$\scDemE$] $\RModal(\A, \B) \equiv \RModal( \fDem(\A), \fDem(\B))$
\end{description}

To express those conditions, we propose three axiom schemata\footnote{Each of these is properly a schema of schemata, since we didn't expand the language $\lMBCL$ with a demodalization operator, which belongs to the metalanguage here. Such linguistic expansion is viable, but it would required to add a semantic clause for the new $d$ symbol, and we are not following that decision in the current work.}:

\begin{description}
	\item[$\aCUNDR$] $(\A \rightarrow \B) \supset (\fDem(\A) \rightarrow \fDem(\B)) \vee (\fDem(\A) \wedge \neg \fDem(\B))$
	\smallskip
	\item[$\aCUNDL$] $(\fDem(\A) \rightarrow \fDem(\B)) \supset (\A \rightarrow \B) \vee (\A \wedge \neg \B)$
	\smallskip
	\item[$\aCUNDE$] $(\A \rightarrow \B) \equiv (\fDem(\A) \rightarrow \fDem(\B)) \vee (\fDem(\A) \wedge \neg \fDem(\B))$
\end{description}

By adding any of those axiom schemata to $\acMBCL$ we can obtain three different axiom systems, which we will denote respectively: $\acMBCLCUNDR$, $\acMBCLCuNDL{}$ and $\acMBCLCUNDE$.

The main advantage of incorporating those axioms is simplification of a class of models for $\MBCL$ containing some specific modal axiom schemata. Let us consider the well-known modal laws: 

\begin{description}
	\item[$\aDOne$] $\Diamond \A \rightarrow \neg \Box \neg \A$
	\smallskip
	\item[$\aDTwo$] $\neg \Box \neg \A \rightarrow \Diamond \A$
	\smallskip
	\item[$\aK$] $\Box(\A \rightarrow \B) \rightarrow (\Box \A \rightarrow \Box \B)$
	\smallskip 
	\item[$\aT$] $\Box\A \rightarrow \A$
	\smallskip
	\item[$\aD$] $\Box\A \rightarrow \Diamond \A$
	\smallskip
	\item[$\aB$] $\A \rightarrow \Diamond \Box\A$
	\smallskip
	\item[$\aFour$] $\Box\A \rightarrow \Box \Box\A$
	\smallskip
	\item[$\aFive$] $\Diamond\A \rightarrow \Box \Diamond\A$
\end{description}

Listed laws are true in a model if following properties of combined frames are fulfilled. Let $\M = \langle \cPos, \rAvail, \eVal, \cRModal \rangle$ and $\oPos, \oPos_{1} \in \cPos$, then:

\begin{description}
	\item[$\scDOne$] $\RModal(\Diamond \A, \neg \Box \neg \A)$
	\smallskip
	\item[$\scDTwo$] $\RModal(\neg \Box \neg \A, \Diamond \A)$
	\smallskip
	\item[$\scKOne$] $\RModal(\Box(\A \rightarrow \B), (\Box \A \rightarrow \Box \B))$
	\smallskip
	\item[$\scKTwo$] $((\rAvail(\oPos, \oPos_{1}) \Rightarrow \R_{\oPos_{1}}(\A, \B)) \Rightarrow \RModal(\Box \A, \Box \B))$
	\smallskip
	\item[$\scT$] $\RModal(\Box \A, \A)\textnormal{ and the modal frame is reflexive}$
	\smallskip
	\item[$\scD$] $\RModal(\Box \A, \Diamond\A)\textnormal{ and the modal frame is serial}$
	\smallskip
	\item[$\scB$] $\RModal(\A, \Box \Diamond\A)\textnormal{ and the modal frame is symmetrical}$
	\smallskip
	\item[$\scFour$] $\RModal(\Box \A, \Box \Box\A)\textnormal{ and the modal frame is transitive}$
	\smallskip
	\item[$\scFive$] $\RModal(\Diamond \A, \Box \Diamond\A)\textnormal{ and the modal frame is Euclidean}$
\end{description}

If $\MBCL$ axiom system contains $\aCUNDR$ schema, mentioned conditions could be simplified in the following way:

\begin{description}
	\item[$\scDOne_d$] $\RModal(\fDem(\Diamond \A), \fDem(\neg \Box \neg \A))$
	\smallskip
	\item[$\scDTwo_d$] $\RModal(\fDem(\neg \Box \neg \A), \fDem(\Diamond \A))$
	\smallskip
	\item[$\scKOne_d$] $\RModal(\fDem(\Box(\A \rightarrow \B)), (\fDem(\Box \A) \rightarrow \fDem(\Box \B)))$
	\smallskip
	\item[$\scKTwo_d$] $((\rAvail(\oPos, \oPos_{1}) \Rightarrow \R_{\oPos_{1}}(\A, \B)) \Rightarrow \RModal(\fDem(\Box \A), \fDem(\Box \B)))$
	\smallskip
	\item[$\scT_d$] $\RModal(\fDem(\Box \A), \A)\textnormal{ and the modal frame is reflexive}$
	\smallskip
	\item[$\scD_d$] $\RModal(\fDem(\Box \A), \fDem(\Diamond\A))\textnormal{ and the modal frame is serial}$
	\smallskip
	\item[$\scB_d$] $\RModal(\A, \fDem(\Box \Diamond\A))\textnormal{ and the modal frame is symmetrical}$
	\smallskip
	\item[$\scFour_d$] $\RModal(\fDem(\Box \A), \fDem(\Box \Box\A))\textnormal{ and the modal frame is transitive}$
	\smallskip
	\item[$\scFive_d$] $\RModal(\fDem(\Diamond \A), \fDem(\Box \Diamond\A))\textnormal{ and the modal frame is Euclidean}$
\end{description}

After application of the demodalization function, listed conditions are of the form:

\begin{description}
	\item[$\scDOne_d$] $\RModal(\A,\neg \neg \A)$
	\smallskip
	\item[$\scDTwo_d$] $\RModal((\neg \neg \A, \A)$
	\smallskip
	\item[$\scKOne_d$] $\RModal((\A \rightarrow \B)), (\A \rightarrow \B))$
	\smallskip
	\item[$\scKTwo_d$] $((\rAvail(\oPos, \oPos_{1}) \Rightarrow \R_{\oPos_{1}}(\A, \B)) \Rightarrow \RModal(\A, \B))$
	\smallskip
	\item[$\scT_d$] $\RModal(\A, \A)\textnormal{ and the modal frame is reflexive}$
	\smallskip
	\item[$\scD_d$] $\RModal(\A, \A)\textnormal{ and the modal frame is serial}$
	\smallskip
	\item[$\scB_d$] $\RModal(\A, \A)\textnormal{ and the modal frame is symmetrical}$
	\smallskip
	\item[$\scFour_d$] $\RModal(\A,\A)\textnormal{ and the modal frame is transitive}$
	\smallskip
	\item[$\scFive_d$] $\RModal(\A, \A)\textnormal{ and the modal frame is Euclidean}$.
\end{description}

Hence, conditions $\scDOne$, $\scDTwo$, $\scT$, $\scD$, $\scB$, $\scFour$, $\scFive$ are reduced to reflexivity of the relating relations and standard conditions for modal frames. In particular, if ($\mathsf{c}$) is any of the mentioned conditions, notice that under $\scDemL$ we have that for ($\mathsf{c}$) to hold it is enough for ($\mathsf{c}$)$_d$ to hold. Therefore ($\mathsf{c}$)$_d$ is a sufficient condition for the relevant frame property to obtain, a condition which is likely easier to check than ($\mathsf{c}$)$_d$. This is why we are going to focus our attention over classes of models closed under $\scDemL$.

For this reason we introduce the axiom system $\acMBCLCuNDL$, which is the Hilbert-style calculus obtained from $\aCPL$, $\aAi$, $\aAii$, $\aBi$, $\aBii$, $\aImp$, $\acDS$, by the addition of schema $\aCUNDL$. Let us denote the class of all the combined frames whose relations satisfy the $\MBCL$ conditions and $\scDemL$ as $\cRelMBCLDEML$. 

\begin{tw}[Soundness Theorem for $\acMBCLCuNDL$]\label{SoundnessMBCLDEML}
	For any $\Gamma \subseteq \fMBCL, \A \in \fMBCL$, $\Gamma \vdash_{\acMBCLCuNDL} \A \Rightarrow$  $\Gamma \vDash_{\cRelMBCLDEML} \A$.
\end{tw}

\begin{proof} By theorem \ref{SoundnessMBCLGCUN}, soundness holds for all the shared cases.
	
	\textit{$\aCUNDL$}. Let $\M, \oPos \vDash d(A) \to d(B)$ and $\M, \oPos \nvDash A \to B$. The latter holds iff either (a) $\M, \oPos \nvDash A \supset B$ or (b) $\sim \R_\oPos (A,B)$. (b) cannot be the case, since $\M \in \cModMBCL\cRelMBCLDEML$, therefore since $\R_\oPos (d(A),D(B))$, by $\scDemL$ also $\R_\oPos (A,B)$. Hence (a) holds.
\end{proof}

For completeness, we perform the same strategy employed in the previous section.

\begin{df}[First canonical Model for $\acMBCLCuNDL$]\label{CanonicalModelMBCLCUNDL}
	
	The first canonical model for $\acMBCLCuNDL$ differs from the one for $\acMBCLGCUN$ (see definition \ref{CanonicalModelMBCLGCUN}) only for the set of possible worlds:
	
	\begin{description}
		\item[(a)] $\cPos = \cMBCLCUNDLMax$
	\end{description}
\end{df}

\begin{df}[Second Canonical Model for $\acMBCLCuNDL$]\label{CanonicalModelMBCLCUNDLsecond}
	
	Let $\M = \langle \cPos, \rAvail, \{\R_{\oPos}\}_{\oPos \in \cPos}$, $\eVal \rangle$ be the first canonical model for $\acMBCLCuNDL$. For $\oPos \in \cPos$, let $\R_\oPos^d \subseteq \fMBCL \times \fMBCL$ be the least relation $\rho$ such that:
	\begin{description}
		\item[(a)] $\M, \oPos \vDash \A \rightarrow \B \Rightarrow \rho (\A, \B)$
		\smallskip
		\item[(b)] $\M, \oPos \vDash d(\A )\rightarrow d(\B)$ and $\M, \oPos \vDash \A \wedge \neg \B \Rightarrow \rho (\A, \B)$
	\end{description}
	
	\noindent The second canonical model for $\acMBCLCuNDL$ is the tuple $\M^d = \langle \cPos, \rAvail$, $\{\R^d_{\oPos}\}_{\oPos \in \cPos}$, $\eVal \rangle$.
\end{df}

In the the definition of the second canonical model, now the defining characteristic of the derived relation $\R^d$ is no longer closure under negation but the role played by the demodalizing function. Point (b) of the definition in fact guarantees the relation of two modalized formulas even when the material implication between them fails but the relating implication between their demodalized versions holds, that is to say when the demodalized versions are related. This is just a rewording of $\scDemL$.

Since the first canonical model for $\acMBCLCuNDL$ is almost identical to the one for $\acMBCLGCUN$, the canonical valuation lemma immediately follows from lemma \ref{CanonicalLemmaMBCLGCUN}.

\begin{lm}\label{CanonicalLemmaMBCLCUNDL}
	Let $\M$ be the first canonical model for $\acMBCLCuNDL$. Then, for all $\oPos \in \cPos_{\M}, \A \in \fMBCL$, $\M, \oPos \vDash \A \Leftrightarrow \A \in \oPos$.
\end{lm}

Again, the two canonical models are semantically indistinguishable. Moreover, the second one belongs to the desired class of models.

\begin{lm}
	Let $\M$ be the first canonical model for $\acMBCLCuNDL$, and $\M^d$ be the second one. Then for all $\oPos \in \cPos, \M, \oPos \vDash A \Leftrightarrow \M^d, \oPos \vDash A$.
\end{lm}

\begin{proof}
	We proceed by induction over the complexity of formulas. The only non-trivial case is the one concerning the connexive arrow, since the two models share the same valuation $\eVal$ and accessibility relation $\rAvail$. 
	
	\textit{Let $\A = \B \rightarrow \C$.} `\textit{$\Rightarrow$}'. Suppose that $\M, \oPos \vDash \B \rightarrow \C$. Then ($\M, \oPos \nvDash \B$ or $\M, \oPos \vDash \C$) and $\R_\oPos(\B,\C)$. By induction hypothesis, $\M^d, \oPos \nvDash \B$ or $\M^d, \oPos \vDash \C$. By definition \ref{CanonicalModelMBCLCUNDLsecond}, $\M, \oPos \vDash \B \rightarrow \C$ implies $\R_\oPos^d(\B,\C)$. We obtain $\M^d, \oPos \vDash \B \rightarrow \C$.
	
	`\textit{$\Leftarrow$}'. Suppose that $\M^d, \oPos \vDash \B \rightarrow \C$. Then ($\M^d, \oPos \nvDash \B$ or $\M^d, \oPos \vDash \C$) and $\R_\oPos^d (\B,\C)$. By definition \ref{CanonicalModelMBCLCUNDLsecond}, $\R^d (\B,\C)$ can hold only if either (a) $\M, \oPos \vDash \B \rightarrow \C$ or (b) $\M, \oPos \vDash d(B) \to d(C)$ and $\M, \oPos \vDash B \land \neg C$. If (a) holds, we are done. If (b) is the case, since $\M^d, \oPos \nvDash \B$ or $\M^d, \oPos \vDash \C$, by induction hypothesis $\M, \oPos \nvDash B$ or $\M, \oPos \vDash C$, which contradicts $\M, \oPos \vDash B \land \neg C$. Therefore (a) must hold.
\end{proof}

\begin{lm}
	Let $\M^d$ be the second canonical model for $\acMBCLCuNDL$. Then $\M^d \in \cModBCL{\cRelMBCLDEML}$.
\end{lm}

\begin{proof}
	Let $\M = \langle \cPos, \rAvail$, $\{\R_{\oPos}\}_{\oPos \in \cPos}$, $\eVal \rangle$. be the first canonical model for $\acMBCLCuNDL$, and $\M^d$ the second one. Assume $\R_\oPos^d(d(A),d(B))$. By definition \ref{CanonicalModelMBCLCUNDLsecond}, this holds only in two cases: either (a) $\M, \oPos \vDash d(A) \rightarrow d(B)$ or (b) $d(A) = A_1, d(B) = B_1, \M, \oPos \vDash d(A_1) \to d(B_1)$ and $\M \vDash A_1 \land \neg B_1$. 
	
	If (a) holds, by properties of maximal consistent sets, $\M, \oPos \vDash d(A) \rightarrow d(B)$ implies $\M, \oPos \vDash (A \rightarrow B) \lor (A \land \neg B)$ by $\aCUNDL$. By the same properties, $\M, \oPos \vDash A \rightarrow B$ or $\M, \oPos \vDash A \land \neg B$. In the former case $\R_\oPos (A,B)$ follows, in the latter as well by definition \ref{CanonicalModelMBCLCUNDLsecond}.
	
	If we consider (b), notice that the demodalization $d$ is obviously an idempotent function, i.e. $d(P) = d(d(P))$. Therefore $d(A) = d(d(A)) = d(A_1)$, and similarly $d(B) = d(B_1)$. Hence $\M, \oPos \vDash d(A_1) \to d(B_1)$ is the same as $\M, \oPos \vDash d(A) \to d(B)$, which is subsumed under case (a).
	
	For all the other properties, the proof is an adaptation to the modal setting of \cite{Klonowski2}, proposition 4.11.
\end{proof}

We can finally prove completeness, since the previous two facts prove that working with the first canonical model is enough to be guaranteed that everything that holds there holds in a model of the class $\cModMBCL\cRelMBCLDEML$.

\begin{tw}[Completeness Theorem for $\acMBCLCuNDL$]\label{completeness_MBCL_CuDL}
	For any $\Gamma \subseteq \fMBCL$, $\A \in \fMBCL$, $\Gamma \vDash_{\cRelMBCLDEML} \A$ $\Rightarrow$ $\Gamma \vdash_{\acMBCLCuNDL} \A$.
\end{tw}

As last result, we prove that $\acMBCLCuNDL$ is a decidable calculus, and as corollary when obtain the decidability of $\acMBCL, \acBCL$ and the other systems introduced in this article as well. In order to do that we use the completeness established in theorem \ref{completeness_MBCL_CuDL} and proceed via the filtration technique (see \cite{modal_logic_blue_book}). We are going to adapt that method to our relating semantics for $\acMBCLCuNDL$ and ultimately prove a form of finite model property. First we define what a set closed under subformulas is.

\begin{df}
	A set $\Gamma \subseteq \fMBCL$ is closed under subformulas if $\forall A,B\in\fMBCL$:
	\begin{itemize}
		\item $A \star B \in \Gamma \Rightarrow A \in \Gamma$ and $B \in \Gamma$, for $\circ \in \lbrace \land, \lor, \to \rbrace$
		
		\item $\ast A \in \Gamma \Rightarrow A \in \Gamma$, for $\circ \in \lbrace \neg, \Box, \diamond \rbrace$
	\end{itemize}
\end{df}

\noindent Notice that when a set of formulas is finite, its closure under subformulas is again finite.

Given a set $\Gamma$ closed under subformulas and a model $\M = \langle \cPos, \rAvail$, $\{\R_{\oPos}\}_{\oPos \in \cPos} \rangle$, we defined the relation $\sim_\Gamma$ $\subseteq \cPos\times\cPos$ as: 

$$\oPos_1 \sim_\Gamma \oPos_2 \Leftrightarrow \forall A,B\in\Gamma, ((v(\oPos_1,A)=1 \Leftrightarrow v(\oPos_2,A)=1) \textnormal{ and }$$$$R_{\oPos_1}(A,B) \Leftrightarrow R_{\oPos_2}(A,B))$$

$\sim_\Gamma$ is an equivalence relation that identifies worlds which are semantically indistinguishable for what it concerns the behaviour of the formulas in $\Gamma$, including how they are related. Now we can quotient the set of worlds and obtain a new model. Let us denote by $[\oPos]_\Gamma$ the equivalence class of $\oPos$ w.r.t. $\sim_\Gamma$. 

\begin{df}[Filtration]\label{filtration_def}
	Let $\M = \langle \cPos, \rAvail$, $\{\R_{\oPos}\}_{\oPos \in \cPos}, \eVal \rangle$ be a model belonging to $\cModBCL{\cRelMBCLDEML}$ and $\Gamma \subseteq \fMBCL$ be closed under subformulas. The filtration of $\M$ through $\Gamma$ is the model $\M^f_\Gamma = \langle \cPos^f, \rAvail^f$, $\{\R^f_{[\oPos]_\Gamma}\}_{[\oPos]_\Gamma \in \cPos^f}, \eVal^f \rangle$ defined as follows:
	\begin{itemize}
		\item $\cPos^f = \{ [\oPos]_\Gamma$ $|$ $\oPos\in\cPos \}$
		
		\item  $[\oPos_1]_\Gamma \rAvail^f [\oPos_2]_\Gamma \Leftrightarrow \exists\oPos_1'\in[\oPos_1]_\Gamma, \oPos_2'\in[\oPos_2]_\Gamma$ s.t. $\oPos_1'\rAvail\oPos_2'$
		
		\item  $\forall[\oPos]_\Gamma \in \cPos^f, \R^f_{[\oPos]_\Gamma} = \{ \langle A,B \rangle \in \R_\oPos$ $|$ $A,B\in\Gamma \}$
		
		\item $\eVal^f([\oPos]_\Gamma,p)=1 \Leftrightarrow \eVal(\oPos,p)=1$, for $p\in\Var$
	\end{itemize}
\end{df}

Both $\R^f$ and $\eVal^f$ are well-defined thanks to the definition of $\sim_\Gamma$. Observe that the filtration $\M^f_\Gamma$ is a model for $\lMBCL$ as in definition \ref{ModelMBCL}, but it is not necessarily closed under demodalization, and in particular if $\Gamma$ is finite then $\R^f$ will be finite as well, therefore not closed under $\scDemL$. This is not a problem, as we will see in a moment. First we need to extend the valuation $\eVal^f$ and prove that $\M$ and $\M^f_\Gamma$ satisfies the same formulas of $\Gamma$.

\begin{lm}\label{filtration_valuation}
	Let $\Gamma \subseteq \fMBCL$ be closed under subformulas. For all $A\in\Gamma,\oPos\in\cPos, \M,\oPos\vDash A \Leftrightarrow \M^f_\Gamma,[\oPos]_\Gamma\vDash A$.
\end{lm}

\begin{proof}
	By induction over the complexity of formulas. The base case for variables is covered by the definition \ref{filtration_def}. The Boolean cases are straightforward.
	
	\textit{Let $A = \Box B$.} $\M,\oPos\vDash \Box B$ holds iff $\forall\oPos'\in\cPos, \oPos\rAvail\oPos'$ implies $\M,\oPos\vDash B$. By induction hypothesis and the fact that $B\in\Gamma$, the latter condition obtains iff $\forall[\oPos]_\Gamma'\in\cPos^f, [\oPos]_\Gamma\rAvail^f[\oPos]'_\Gamma$ implies $\M^f_\Gamma,[\oPos]'_\Gamma\vDash B$. Notice that in the right-to-left direction we are able to cover all the successors of $\oPos$ thanks to the definition of $\sim_\Gamma$. Finally $\forall[\oPos]_\Gamma'\in\cPos^f, [\oPos]_\Gamma\rAvail^f[\oPos]'_\Gamma$ implies $\M^f_\Gamma,[\oPos]'_\Gamma\vDash B$ holds iff $\M^f_\Gamma,[\oPos]\vDash \Box B$.
	
	\vspace{0.2cm}
	
	\textit{Let $A = B \rightarrow C$.} $\M,\oPos\vDash B \rightarrow C$ holds iff ($\M,\oPos\nvDash B$ or $\M,\oPos\vDash C$) and $\R_\oPos(B,C)$. By induction hypothesis, ($\M,\oPos\nvDash B$ or $\M,\oPos\vDash C$) iff ($\M^f_\Gamma,[\oPos]_\Gamma\nvDash B$ or $\M^f_\Gamma,[\oPos]_\Gamma\vDash C$). By hypothesis $B,C\in\Gamma$, therefore by definition \ref{filtration_def} we have $\R_\oPos(B,C)$ iff $\R^f_{[\oPos]_\Gamma}(B,C)$. We conclude that ($\M^f_\Gamma,[\oPos]_\Gamma\nvDash B$ or $\M^f_\Gamma,[\oPos]_\Gamma\vDash C$) and $\R^f_{[\oPos]_\Gamma}(B,C)$ iff $\M^f_\Gamma,[\oPos]_\Gamma\vDash B \to C$.
\end{proof}

Now we are ready to address the problem that $\M^f_\Gamma$ doesn't have to belong to $\cModBCL{\cRelMBCLDEML}$. In fact for every filtration $\M^f_\Gamma$ it is possible to build a model $\M^+ \in\cModBCL{\cRelMBCLDEML}$ which is equivalent as it concerns the formulas $A\in\Gamma$.

\begin{df}\label{filtration_def_m+}
	Let $\M = \langle \cPos, \rAvail$, $\{\R_{\oPos}\}_{\oPos \in \cPos}, \eVal \rangle \in \cModBCL{\cRelMBCLDEML}$, let $\Gamma \subseteq \fMBCL$ be closed under subformulas and $\M^f_\Gamma = \langle \cPos^f, \rAvail^f$, $\{\R^f_{[\oPos]_\Gamma}\}_{[\oPos]_\Gamma \in \cPos^f}, \eVal^f \rangle$ the filtration of $\M$ through $\Gamma$. $\M^+_\Gamma = \langle \cPos^f, \rAvail^f$, $\{\R^+_{[\oPos]_\Gamma}\}_{[\oPos]_\Gamma \in \cPos^f}, \eVal^f \rangle$ differs from $\M^f_\Gamma$ only for:
	\begin{itemize}  
		\item  $\forall[\oPos]_\Gamma \in \cPos^f, \R^+_{[\oPos]_\Gamma} = \R^f_{[\oPos]_\Gamma} \cup \{ \langle A,B \rangle$ $|$ $\langle d(A),d(B) \rangle\in\R^f_{[\oPos]_\Gamma} \}$
	\end{itemize}
	\noindent $\M^+_\Gamma$ is closed under $\scDemL$, therefore it belongs to $\cModBCL{\cRelMBCLDEML}$.
\end{df}

\begin{lm}\label{filtration_equivalence_m+}
	Let $\Gamma \subseteq \fMBCL$ be closed under subformulas. For all $A\in\Gamma,\oPos\in\cPos, \M^f_\Gamma,[\oPos]_\Gamma\vDash A \Leftrightarrow \M^+_\Gamma,[\oPos]_\Gamma\vDash A$.
\end{lm}

\begin{proof}
	By induction over the complexity of formulas. Since $\M^f_\Gamma$ and $\M^+_\Gamma$ differ only for the relating relation, the only interesting case is that for $\to$.
	
	\textit{Let $A = B \to C$.} Suppose $\M^+_\Gamma,[\oPos]_\Gamma\vDash B \to C$, then ($\M^+_\Gamma,[\oPos]_\Gamma\nvDash B$ or $\M^+_\Gamma,[\oPos]_\Gamma\vDash C$) and $\R^+(B,C)$. By induction hypothesis, also $\M^f_\Gamma,[\oPos]_\Gamma\nvDash B$ or $\M^f_\Gamma,[\oPos]_\Gamma\vDash C$. Since $B,C\in\Gamma$, then $\R^+(B,C)$ implies $\R^f(B,C)$. In fact if $\sim\R^f(B,C)$, by definition of $\R^+$ we have $\R^f(d(B),d(C))$, hence $\R(d(B),d(C))$ by definition \ref{filtration_def}. Since $\M$ is closed under $\scDemL$, it holds $\R(B,C)$ as well, but $B,C\in\Gamma$, so $\R^f(B,C)$, contradiction. Therefore it must be $\R^f(B,C)$. We conclude that $\M^f_\Gamma,[\oPos]_\Gamma\vDash B \to C$. The other direction is immediate, since $\R^f_{[\oPos]_\Gamma} \subseteq \R^+_{[\oPos]_\Gamma}$.
\end{proof}

We are left to prove that when $\Gamma$ is finite the model $\M^f$ is finite.

\begin{lm}\label{filtration_finite_lemma}
	Let finite $\Gamma \subset \fMBCL$ be closed under subformulas. Given a model $\M$, its filtration $\M^f_\Gamma$ is a finite model. 
\end{lm}

\begin{proof}
	Let us consider a function $g: \cPos^f \to \mathcal{P}(\Gamma)$, where $\mathcal{P}(\Gamma)$ is the powerset of $\Gamma$, defined by $g([\oPos]_\Gamma) = \{ A\in\Gamma$ $|$ $\M^f_\Gamma,[\oPos]_\Gamma\vDash A \}$. $g$ is injective, in fact if $g([\oPos_1]_\Gamma)=g([\oPos_2]_\Gamma)$ it means that $\forall A\in\Gamma, \M^f_\Gamma,[\oPos_1]_\Gamma\vDash A$ iff $\M^f_\Gamma,[\oPos_2]_\Gamma\vDash A$, that is $\oPos_1 \sim_\Gamma \oPos_2$, which amounts to $[\oPos_1]_\Gamma = [\oPos_2]_\Gamma$. Therefore, if we denote by $|.|$ the cardinality of a set, we have $|\cPos^f| \leq |\mathcal{P}(\Gamma)| = 2^{|\Gamma|}$.
	
	\noindent Furthermore $\{\R^f_{[\oPos]_\Gamma}\}_{[\oPos]_\Gamma \in \cPos^f}$ is a finite collection of relations, since $\cPos^f$ is finite, and each $\R^f_{[\oPos]_\Gamma}$ is finite, since $\langle A,B \rangle \in \R^f$ only if $A,B\in\Gamma$, which is a finite set of formulas. 
\end{proof}

\begin{tw}\label{filtration_finite_model_property}
	If a formula $A \in \fMBCL$ is satisfiable in $\cRelMBCLDEML$, it is satisfiable in a finite model which is equivalent w.r.t. to $A$ to a model belonging to $\cModMBCL{\cRelMBCLDEML}$.
\end{tw}

\begin{proof}
	Let $A$ be satisfied by $\M \in \cModMBCL{\cRelMBCLDEML}$. Consider the closure of $A$ under subformula $\Gamma$ and the filtration $\M^f_\Gamma$. By lemmas \ref{filtration_valuation} and \ref{filtration_finite_lemma}, $A$ is satisfied by $\M^f_\Gamma$, which is a finite model. By lemma \ref{filtration_equivalence_m+}, $\M^f_\Gamma$ is equivalent w.r.t. $A$ to $\M^+_\Gamma \in \cModMBCL{\cRelMBCLDEML}$ as defined in definition \ref{filtration_def_m+}.
\end{proof}

\begin{wn}
	The logic $\acMBCLCuNDL$ is decidable.
\end{wn}

\begin{proof}
	It follows from the completeness theorem \ref{completeness_MBCL_CuDL} and the finite model property of theorem \ref{filtration_finite_model_property}.
\end{proof}

A simpler version of this proof can be adapted in order to prove the decidability of $\acMBCL$. The definition of filtration for a model $\M \in \cModBCL{\cRelMBCL}$ is the same as definition \ref{filtration_def}. Moreover for the filtration $\M^f_\Gamma$ it already holds that $\M^f_\Gamma \in \cModBCL{\cRelMBCL}$, therefore there is no need to build a model $\M^+_\Gamma$ and to prove a corresponding result of lemma \ref{filtration_equivalence_m+}. With this simplification, the above procedure can be repeated, obtaining the finite model property (this time in a strict sense, not like in theorem \ref{filtration_finite_model_property}), which together with completeness theorem \ref{completeness_basic_mbcl} gives decidability for $\acMBCL$. Moreover, since $\acMBCL$ is a conservative expansion of $\acBCL$, the completeness of the former implies the completeness of the latter, it is enough to work on the $\Box$-free fragment of $\acMBCL$, that is precisely $\acBCL$. Finally, $\acBCLGCUN$ and $\acMBCLGCUN$ are axiomatic extensions of decidable logics, therefore they inherit the decidability. To sum up these considerations:

\begin{wn}
	The logics $\acBCL, \acBCLGCUN, \acMBCL, \acMBCLGCUN$ are decidable.
\end{wn}

\section{Further work: a syncategorematic approch to negation and modalities}
\label{sec:6}

The results presented in this work are meant as a preliminary study in the direction of the more general picture of relating logics closed under arbitrary negations and demodalization. Within the case study of connexivity, we have proved that it is possible to strengthen the logics $\BCL$ and $\MBCL$ in order to obtain closures under special cases of iterated negations and under certain applications of the demodalization function. The decision to work in the setting of connexive logic was motivated by the fact that it was precisely during its study in \cite{JarmuzekMalinowski2019} that it emerged the peculiar connection between the relating relation and the behaviour of iterated negation.

As we have claimed in the current paper, there is much work to be done. The two kinds of closures studied here have a deeper, non-technical motivation, and that is syncategorematicity. Closure under multiple negations expresses, up to a certain iteration of the negation connective, that this connective does not alter part of the intensional component of a sentence. Here we presented only particular cases of such closure, but it is an ongoing project the attempt to formulate a general theory of closure under multiple negations, either within Boolean Connexive Logics and, to the highest degree of generality, for relating semantics.

Similarly, the possibility to preserve part of the intensional content of a sentence through the process of demodalization (or modalization), is a way to clarify how the relation between formulas expands from (or to) their modal counterparts. This is an application of the syncategorematic approach to modality. While we considered left demodalization $\scDemL$, there are still other directions of demodalization to be pursued. After that, partial demodalization is the next step: what if the demodalization function operates only up to or from a certain complexity? As shown, the effect of demodalization has important effects on frame conditions, which are a key element in the study of modal logics. Different forms of demodalization may allow to find a more precise connection with said conditions. The attempt of our paper was to lay an initial framework for exploring these questions.

\paragraph{Acknowledgements.} The authors would like to thank the anonymous reviewers for their comments and questions. Most of the comments and questions were relevant and allowed us to improve our article.



\AuthorAdressEmail{Tomasz Jarmu\.zek}{Department of Logic\\
	Nicolaus Copernicus University in Toru\'n\\
	Moniuszki 16\\
	63-300 Toru\'n, Poland}{Tomasz.Jarmuzek@umk.pl}

\AdditionalAuthorAddressEmail{Jacek Malinowski}{Institute of Philosophy and Sociology\\Polish Academy of Sciences\\
	Nowy Świat 72 \\
	00-330  Warsaw, Poland}{jacek.malinowski@StudiaLogica.org}

\AdditionalAuthorAddressEmail{Aleksander Parol}{Department of Logic\\
	Nicolaus Copernicus University in Toru\'n\\
	Moniuszki 16\\
	63-300 Toru\'n, Poland}{aleksander.parol@doktorant.umk.pl}

\AdditionalAuthorAddressEmail{Nicol\`o Zamperlin}{Department of Pedagogy, Psychology and Philosophy\\
	University of Cagliari\\
	Cagliari, Italy}
{n.zamperlin@gmail.com}

\end{document}